\documentclass[a4paper,12pt]{article}
\usepackage[top=3cm, left=2.5cm, bottom=3cm, right=2.5cm]{geometry}

\usepackage{geometry}
\usepackage{graphicx}
\usepackage{amsmath, amsthm, amssymb} 
\usepackage{overpic}
\usepackage{dsfont}
\usepackage{color}
\usepackage{tikz}

\RequirePackage[colorlinks, hyperindex, plainpages=false]{hyperref} 
\usepackage[ngerman, french, english]{babel}
\usepackage[latin1]{inputenc}
\usepackage[T1]{fontenc}
\usepackage{color}
\usepackage{setspace}
\usepackage{longtable}
\usepackage{bm}

\newtheorem{theorem}{Theorem}[section]
\newtheorem{proposition}[theorem]{Proposition}

\newtheorem{corollary}[theorem]{Corollary}

\newtheorem{lemma}[theorem]{Lemma}

\newtheorem{remarks}[theorem]{Remarks}

\newtheorem{assumption}[theorem]{Assumption}

\DeclareMathOperator\dist{dist}

\begin{document}

\title{Large Deviation Principle for Reflected Poisson driven $SDE$s in Epidemic Models.}

\author{
{Etienne Pardoux}\footnote{{Aix--Marseille Univ, CNRS, Centrale Marseille, I2M, 
Marseille, France; etienne.pardoux@univ-amu.fr; brice.samegni-kepgnou@univ-amu.fr.}}
\and
{Brice Samegni-Kepgnou}\setcounter{footnote}{6}$^\ast$
}

\maketitle

\begin{abstract}
We establish a large deviation principle for a reflected Poisson driven $SDE$. Our motivation is to  study in a forthcoming paper the problem of exit of such a process from the basin of attraction of a locally stable equilibrium associated with its law of large numbers. Two examples are described in which we verify  the assumptions that we make to establish  the large deviation principle. 
\end{abstract}


\vskip 3mm

\section{Introduction}

 Consider a compartmental model for an infectious disease which takes the form of the Poisson driven $SDE$ 
\begin{equation}\label{PoissonSDE}
Z^{N,z}(t):=\frac{[Nz]}{N}+\frac{1}{N}\sum_{j=1}^{k}h_{j}P_{j}\Big(N\int_{0}^{t} \beta_{j}(Z^{N}(s)) ds \Big), 
\end{equation}
where $N$ is the total size of the population that is assumed to be constant, $k$ is the number of possible types of transitions for a given individual, the $d$ components of $Z^{N,z}(t)$ denote the proportions of individuals in the $d$ distincts compartments at time $t$, $P_{j}$ $(j=1,...,k)$ are  mutually independent standard Poisson processes, $h_{j}\in\{-1,0,1\}^{d}$ $(j=1,...,k)$ are the distinct jump directions with rates $\beta_{j}(.)$ which  are $\mathbb{R}_{+}-$valued, and such that the solution of \eqref{PoissonSDE} remains in the set
\begin{equation*}
A:=\Big\{z\in\mathbb{R}_{+}^{d}: \sum_{i=1}^{d}z_{i}\leq1\Big\}.
\end{equation*}
Under appropriate assumptions, 
as $N\to\infty$, $Z^{N,z}(t)\to Y^z(t)$ in probability, locally uniformly in $t$, where $Y^z(t)$ solves the ODE
\begin{equation}\label{ODE}
\frac{dY^z}{dt}(t)=b(Y^z(t)),\ Y^z(0)=z,
\end{equation}
with $b(z)=\sum_{j=1}^k\beta_j(z)h_j$.
We have established in \cite{Sam2017} a large deviation principle for the solution of such an $SDE$. 

Our aim in this paper is to establish a large deviation principle for a similar Poisson driven $SDE$, which is reflected in the interior of some subset $O$ of $A$. Indeed, let us look at the literature on the exit problem from a domain $O$ of a dynamical system with a small Brownian perturbation. We note that most authors assume that the dynamical system crosses $\partial O$ non-tangentially i.e. $\langle n(z), b(z)\rangle<0$, for all $z\in\partial O$ where $n(z)$ is the unit outward normal to $\partial O$ at $z$. On the other hand,  Day in \cite{day1990large} considers the case of a "characteristic boundary", i.e. where the dynamical system which begins at a point on the boundary remains there, in other words  $\langle n(z), b(z)\rangle=0$ for all $z\in \partial O$. For instance, $O$ might be the basin of attraction of a locally stable equilibrium of the ODE \eqref{ODE}. In order to study the limit as $N\to\infty$ of the exit law of a small Brownian perturbation of the solution of \eqref{ODE} from $O$, he needs first to study the large deviation of a reflected $SDE$.
 
 Similarly, we want to study in a forthcoming publication the large deviation of the exit law of a solution of \eqref{PoissonSDE} through a characteristic boundary  of a stable domain of its law of large numbers limits $Y^z(t)$. 
 Our motivation is the study of the most probable trajectory of $Z^{N,z}(t)$ when exiting the basin of attraction of an endemic situation, which is a locally stable equilibrium of the deterministic model \eqref{ODE}.

The proof of the large deviation principle for the reflecting Poisson driven $SDE$ follows the same steps as the proof of the large deviation principle for the original Poisson driven $SDE$ defined by \eqref{PoissonSDE}, but the arguments are modified when necessary. The paper is organized as follows. The reflected Poisson driven SDE is defined in section \ref{sec2}. The law of large numbers for that reflected SDE is established in section \ref{weakLLN}. Some preliminary results towards the large deviation are established in section \ref{sec4}.
The lower bound is established in section \ref{Seclower}, and the upper bound in section \ref{Secupper}. Finally, in section \ref{sec7}, we show that our assumptions are satisfied in two examples of epidemics models which we have in mind among the possible applications of our results.

\section{The reflected Poisson driven SDE}\label{sec2}
We will need to consider in this paper the sets $\bar{O}^{N}=\bar{O}\cap A^{(N)}$ where $$A^{(N)}=\Big\{z\in A:Nz\in\mathbb{Z}^{d}_{+}\Big\}.$$
For any $z\in\bar{O}$, we let $[N z]=([N z_1],...,[N z_d])$ where $[a]$ denotes the integer part of the real number $a$. And we define the vector $z^{N}$ by
\begin{equation*}
z^{N}=
\begin{cases}
      \frac{[N z]}{N}& \text{if}~~~ \frac{[N z]}{N}\in\bar{O}, \\
     \underset{y\in\bar{O}^{(N)}}{\arg\inf}|y-z|& \text{otherwise}.
\end{cases}
\end{equation*}
 We now define  the $d$--dimensional reflected process $\tilde{Z}^{N,z}_{t}$ by
\begin{align}\label{reflexionsde2}
\tilde{Z}^{N}(t)=\tilde{Z}^{N,z}(t)&:=z^{N}+\frac{1}{N}\sum_{j=1}^{k}h_{j}Q_{t}^{N,j}-\frac{1}{N}\sum_{j=1}^{k}h_{j}\int_{0}^{t}\mathbf{1}_{\{\tilde{Z}^{N}(s-)+\frac{h_{j}}{N}\not\in\bar{O}\}}dQ_{s}^{N,j} \nonumber \\
&:=z^{N}+\frac{1}{N}\sum_{j=1}^{k}h_{j}\int_{0}^{t}\mathbf{1}_{\{\tilde{Z}^{N}(s-)+\frac{h_{j}}{N}\in\bar{O}\}}dQ_{s}^{N,j}, 
\end{align}
where the $Q_{t}^{N,j}$'s are defined by
\begin{equation}\label{Qaux}
Q_{t}^{N,j}=P_{j}\Big(N\int_{0}^{t} \beta_{j}(\tilde{Z}^{N}(s)) ds \Big).
\end{equation}
We remark that for any $t>0$, $\tilde{Z}^{N}(t)\in\bar{O}^{(N)}$ and then a solution of \eqref{reflexionsde2} on the time interval $[0,T]$  belongs $a.s.$ to $D_{T,\bar{O}}$, which is the set of functions from $[0,T]$ into $\bar{O}$ and which are right continuous and have left limits. The aim of this paper is to show that the solution of the above reflecting Poisson driven $SDE$ obeys the same large deviation principle with the same "good" rate function as the solution of \eqref{PoissonSDE}. In the sequel, we denote $\mathbb{P}^{N}_{z}$ the probability measure on $D_{T,\bar{O}}$ such that the process $\tilde{Z}^{N}$ has as initial condition $\tilde{Z}^{N}(0)=z^{N}$, where $z^{N}$ is chosen as we specified above.

 The main difficulty to establish our result is that  some of the rates $\beta_{j}$ vanish on parts of the boundary of the set $O$. To solve this problem, we make the following assumptions:

\begin{assumption}\label{assumpinq}
\begin{enumerate}
\item $\bar{O}$ is compact and there exists a point $z_{0}$ in the interior of $O$ such that each segment joining $z_{0}$ and any $z\in\partial O$ does not touch any other point of the boundary $\partial O$.  \label{assump21}
\item For each $a>0$ small enough and $z\in\bar{O}$, denoting $z^{a}=z+a(z_{0}-z)$, we assume that there exist two positive constants $c_{1}$ and $c_{2}$ such that  \label{assump22}
 \begin{align*}
   |z-z^{a}| &\leq c_{1} a,  \\
   \dist(z^{a},\partial O) &\geq c_{2} a.  
\end{align*}
 \item The rate functions $\beta_{j}$ are Lipschitz continuous with the Lipschitz constant $C$.  \label{assump23}
 \item There exist two constants $\lambda_{1}$ and $\lambda_{2}$ such that, whenever $z\in \bar{O}$ is such that $\beta_{j}(z)<\lambda_{1}$, $\beta_{j}(z^{a})>\beta_{j}(z)$ for all $a\in]0, \lambda_{2}[$ . \label{assump24}
 \item There exist constants $\nu\in]0,1/2[$ and $a_0>0$ such that  \label{assump25}
$C_a\geq\exp\{-a^{-\nu}\}$ for all $a<a_0$,
 where
 \begin{equation*}
 C_{a}=\inf_{j}\inf_{z:\dist(z,\partial O)\geq c_{2}a}\beta_{j}(z).
 \end{equation*}
 \end{enumerate}
\end{assumption}
We define  $\sigma=\underset{1\leq j\leq k}{\sup}\quad\underset{z\in A}{\sup}\beta_{j}(z)$.
\section{The weak law of large numbers}\label{weakLLN}

 We first assume that $\partial O$ is smooth enough so that the following assumption is satisfied. 
\begin{assumption}\label{assump3}
There exists a function $u\in C^{1}_{b}(\bar{O})$ which satisfies the following assumptions:
\begin{enumerate}
  \item $O=A\cap\{z\in\bar{O}: u(z)>0\}$, $\partial O=A\cap\{z\in\bar{O}: u(z)=0\}$. \label{assump31}
  \item $\nabla u(z)\neq0$ for all $z\in\partial O$.  \label{assump32}
  \item There exist $C_{1}$, $C_2>0$ such that $\min\{C_{1}\dist(z,\partial O), C_2\}\leq u(z)$, for all $z\in\bar{O}$ \label{assump33}
  \item $\big<b(z), \nabla u(z)\big>\geq0$ for all $z\in\bar{O}$, with $b(z)=\underset{j=1}{\overset{k}{\sum}}\beta_j(z)h_j$. \label{assump34}
  \item There exits $\rho>0$ such that $\big<-g_{N}(z), \nabla u(z)\big>\geq\rho\underset{j=1}{\overset{k}{\sum}}\mathbf{1}_{\big\{ z+\frac{h_{j}}{N}\not\in\bar{O}\big\}}$ for all $z\in\bar{O}$, \label{assump35}
\end{enumerate}
where 
\begin{equation*}
g_{N}(z)=\sum_{j=1}^{k}\mathbf{1}_{\big\{ z+\frac{h_{j}}{N}\not\in\bar{O}\big\}}h_{j}\beta_{j}(z).
\end{equation*}
\end{assumption}

 In the remaining of this section we assume that both Assumption \ref{assumpinq} and Assumption \ref{assump3} are in force.
\begin{lemma}\label{le:1}
Let $(\widetilde{\mathcal{M}}^{N}(t))_{t\geq 0}$ be the process defined for all $t\geq0$ by
\begin{equation*}
\widetilde{\mathcal{M}}^{N}(t)=\sum_{j=1}^{k}\int_{0}^{t}h_{j}\mathbf{1}_{\big\{ \tilde{Z}^{N}(s^{-})+\frac{h_{j}}{N}\in\bar{O}\big\}}d\widetilde{Q}^{N,j}_{s},
\end{equation*}
where 
\begin{equation*}
\widetilde{Q}^{N,j}_{s}=\frac{1}{N}Q^{N,j}_{s}-\int_{0}^{s}\beta_{j}(\tilde{Z}^{N}(r))dr.
\end{equation*}
Then  $\widetilde{\mathcal{M}}^{N}(.)$ is a square integrable martingale and for all $T>0$,
\begin{equation*}
\sup_{0\leq t\leq T}|\widetilde{\mathcal{M}}^{N}(t)|\underset{N \to +\infty}{\overset{\mathbb{P}}{\longrightarrow}}0.
\end{equation*}

\end{lemma}

\begin{proof}
 $\widetilde{\mathcal{M}}^{N}(t)$ is a square integrable martingale since it is a sum of $k$ stochastic integrals of bounded predictable processes with respect to square integrable martingales. We deduce from Doob's inequality, see e.g. \cite{revuz2013} ,that
\begin{align*}
\mathbb{E}\Big(\sup_{0\leq t\leq T}|\widetilde{\mathcal{M}}^{N}(t)|^{2}\Big)&\leq 4 \mathbb{E}\Big(|\widetilde{\mathcal{M}}^{N}(T)|^{2}\Big)\\
&= 4 \mathbb{E}\Big(\big<\widetilde{\mathcal{M}}^{N}\big>_{T}\Big),
\end{align*}
where $\big<\mathcal{M}^{N}\big>_t$ is the increasing predictable process such that $|\mathcal{M}^{N}(t)|^{2}-\big<\mathcal{M}^{N}\big>_t$ is a martingale. We have
\begin{align*}
   \big<\widetilde{\mathcal{M}}^{N}\big>_{T} &=\int_{0}^{T}\sum_{j=1}^{k}|h_{j}|^{2}\mathbf{1}_{\big\{ \tilde{Z}^{N}(t^{-})+\frac{h_{j}}{N}\in\bar{O}\big\}}\frac{1}{N^{2}}N\beta_{j}(\tilde{Z}^{N}(t))dt   \\
    &\leq \frac{k d \sigma T}{N}. 
\end{align*}
Thus
\begin{equation*}
\mathbb{E}\Big(\sup_{0\leq t\leq T}|\widetilde{\mathcal{M}}^{N}(t)|^{2}\Big)\leq \frac{k d \sigma T}{N}\to0,\quad\text{as}\quad N\to\infty.
\end{equation*}
The result follows.
\end{proof}

 We note that equation \eqref{reflexionsde2} can be re-written as
\begin{equation}\label{reflexion2}
\tilde{Z}^{N}(t)=z^{N}+\int_{0}^{t}b(\tilde{Z}^{N}(s))ds-\int_{0}^{t}g_{N}(\tilde{Z}^{N}(s))ds+\widetilde{\mathcal{M}}^{N}(t),
\end{equation}
where for all $z\in A$, $b(z)=\underset{j=1}{\overset{k}{\sum}}\beta_{j}(z)h_{j}$.
\begin{lemma}\label{le:2}
Let $(\widetilde{\mathcal{X}}^{N}(t))_{t\geq0}$ be the process defined for all $t\geq0$ by
\begin{equation*}
\widetilde{\mathcal{X}}^{N}(t)=\int_{0}^{t}\sum_{j=1}^{k}\mathbf{1}_{\big\{ \tilde{Z}^{N}(s)+\frac{h_{j}}{N}\not\in\bar{O}\big\}}ds.
\end{equation*}
Then for all $T>0$,
\begin{equation*}
\widetilde{\mathcal{X}}^{N}(T)\underset{N \to +\infty}{\overset{\mathbb{P}}{\longrightarrow}}0.
\end{equation*}
\end{lemma}

\begin{proof}
 Let $u$ be the function appearing in the Assumption \ref{assump3}. Applying It\^o's formula to $u$, we deduce that,
\begin{align*}
   u(\tilde{Z}^{N}(t)) =&u(z^{N})+\int_{0}^{t}\big<\nabla u(\tilde{Z}^{N}(s)), b(\tilde{Z}^{N}(s))\big>ds-\int_{0}^{t}\big<\nabla u(\tilde{Z}^{N}(s)), g_{N}(\tilde{Z}^{N}(s))\big>ds   \\
    +&\int_{0}^{t}\big<\nabla u(\tilde{Z}^{N}(s)), d\widetilde{\mathcal{M}}^{N}(s)\big> 
    +\sum_{s\leq t}\big[u(\tilde{Z}^{N}(s))-u(\tilde{Z}^{N}(s^{-}))-\big<\nabla u(\tilde{Z}^{N}(s^{-})), \Delta\tilde{Z}^{N}(s)\big>\big].
\end{align*}
Thus, we can use the Assumption \ref{assump3} \ref{assump34} to deduce
\begin{equation}\label{ineqle1}
   u(\tilde{Z}^{N}(t)) \geq u(z^{N})-\int_{0}^{t}\big<\nabla u(\tilde{Z}^{N}(s)), g_{N}(\tilde{Z}^{N}(s))\big>ds+\tilde{\mathcal{\zeta}}^{N,1}_{t}+\tilde{\mathcal{\zeta}}^{N,2}_{t},
\end{equation}
where
\begin{align*}
    \tilde{\mathcal{\zeta}}^{N,1}_{t}&=\int_{0}^{t}\big<\nabla u(\tilde{Z}^{N}(s)), d\widetilde{\mathcal{M}}^{N}(s)\big>    \\
    \tilde{\mathcal{\zeta}}^{N,2}_{t}&=\sum_{s\leq t}\big[u(\tilde{Z}^{N}(s))-u(\tilde{Z}^{N}(s^{-}))-\big<\nabla u(\tilde{Z}^{N}(s^{-})), \Delta\tilde{Z}^{N}(s)\big>\big].  
\end{align*}
Moreover $\underset{0\leq t\leq T}{\sup}|\tilde{\mathcal{\zeta}}^{N,1}_{t}|\underset{N \to +\infty}{\overset{\mathbb{P}}{\longrightarrow}}0$. Indeed, again from Doob's inequality, 
\begin{align*}
    \mathbb{E}\Big(\sup_{0\leq t\leq T}|\tilde{\mathcal{\zeta}}^{N,1}_{t}|^{2}\Big)&\leq4\mathbb{E}\big(|\tilde{\mathcal{\zeta}}^{N,1}_{T}|^{2}\big)   \\
    &\leq4\mathbb{E}\Big(\big<\tilde{\mathcal{\zeta}}^{N,1}\big>_{T}\Big).  
\end{align*}
But
\begin{align*}
    \big<\tilde{\mathcal{\zeta}}^{N,1}\big>_{T}&=\int_{0}^{T}|\nabla u(\tilde{Z}^{N}(t))|^{2}d\big<\widetilde{\mathcal{M}}^{N}\big>_{t}   \\
    &=\frac{1}{N}\sum_{j=1}^{k}|h_{j}|^{2}\int_{0}^{T} |\nabla u(\tilde{Z}^{N}(t))|^{2}\mathbf{1}_{\big\{ \tilde{Z}^{N}(t)+\frac{h_{j}}{N}\not\in\bar{O}\big\}}\beta_{j}(\widetilde{Z}^{N}(t))dt \\
    &\leq\frac{k d \sigma K_{1} T}{N}.
\end{align*}
Then
\begin{equation*}
\mathbb{E}\Big(\sup_{0\leq t\leq T}|\tilde{\mathcal{\zeta}}^{N,1}_{t}|^{2}\Big)\leq\frac{k d \sigma K_{1} T}{N}\to0\quad\text{as}\quad N\to\infty.
\end{equation*}
We also have $\underset{0\leq t\leq T}{\sup} |\tilde{\mathcal{\zeta}}^{N,2}_{t}|\underset{N \to +\infty}{\overset{\mathbb{P}}{\longrightarrow}}0$. Indeed by Taylor's expansion, if $0\leq t\leq T$ is a jump time of $\tilde{Z}^{N}(t)$,
\begin{equation*}
u(\tilde{Z}^{N}(t))-u(\tilde{Z}^{N}(t^{-}))=\Big<\nabla u\big(\tilde{Z}^{N}(t^{-})+\theta\Delta\tilde{Z}^{N}(t)\big), \Delta\tilde{Z}^{N}(t)\Big>\  \text{ for some random } 0\leq\theta\leq1.
\end{equation*} 
Consequently, since $u\in C^{1}_{b}(\bar{O})$ and $|\Delta \tilde{Z}^{N}(t)|\le\sqrt{d}/N$,
\begin{align*}
   & |u(\tilde{Z}^{N}(t))-u(\tilde{Z}^{N}(t^{-}))-\big<\nabla u(\tilde{Z}^{N}(t^{-})), \Delta\tilde{Z}^{N}(t)\big>|  \\
   &=\Big |\Big<\nabla u\big(\tilde{Z}^{N}(t^{-})+\theta\Delta\tilde{Z}^{N}(t)\big)-\nabla u\big(\tilde{Z}^{N}(t^{-})\big), \Delta\tilde{Z}^{N}(t)\Big> \Big | \\
    &\leq\Big |\nabla u\big(\tilde{Z}^{N}(t^{-})+\theta\Delta\tilde{Z}^{N}(t)\big)-\nabla u\big(\tilde{Z}^{N}(t^{-})\big)\Big |\times|\Delta \tilde{Z}^{N}(t)| \\
    &\leq \frac{\sqrt{d}}{N}\sup_{z\in\bar{O}, |\theta\Delta z|\leq\frac{\sqrt{d}}{N}}|\nabla u(z+\theta\Delta z)-\nabla u(z)|:= \frac{\sqrt{d}}{N}\delta_{N}, \text{with}\  \delta_{N}\to0\  \text{as}\  N\to\infty.
\end{align*}
 It follows that
\begin{align*}
   \sup_{0\leq t\leq T} |\tilde{\mathcal{\zeta}}^{N,2}_{t}|&= \sup_{t\leq T}\big|\sum_{s\leq t}\big[u(\tilde{Z}^{N}(s))-u(\tilde{Z}^{N}(s^{-}))-\big<\nabla u(\tilde{Z}^{N}(s^{-})), \Delta\tilde{Z}^{N}(s)\big>\big]\big|  \\
    &\leq \sum_{s\leq T}\big |u(\tilde{Z}^{N}(s))-u(\tilde{Z}^{N}(s^{-}))-\big<\nabla u(\tilde{Z}^{N}(s^{-})), \Delta\tilde{Z}^{N}(s)\big>\big | \\
    &\leq \frac{\sqrt{d}}{N}\delta_{N}\sum^{k}_{j=1}Q^{N,j}_{T}
\end{align*}
Then 
\begin{align*}
\mathbb{E}\Big(\sup_{0\leq t\leq T}|\tilde{\mathcal{\zeta}}^{N,2}_{t}|\Big)&\leq\sqrt{d}\frac{1}{N}\mathbb{E}\Big(\sum^{k}_{j=1}Q^{N,j}_{T}\Big)\delta_{N}  \\
&\leq C'k\delta_{N}.
\end{align*}
Let $\delta>0$ and 
\begin{equation*}
\tilde{\mathcal{B}}^{N}=\Big\{z\in\bar{O}: \sum_{j=1}^{k}\mathbf{1}_{\big\{ z+\frac{h_{j}}{N}\not\in\bar{O}\big\}}>0\Big\}
\end{equation*}
and with the convention that $\inf\emptyset=\infty$, let $T^{N,1}_{\delta}$ be the stopping time defined by
\begin{equation*}
    T^{N,1}_{\delta}:= \inf\Big\{t>0: \widetilde{\mathcal{X}}^{N}(t)\geq2\delta/3\Big\}\land T.
\end{equation*}
By using \eqref{ineqle1} and the Assumption \ref{assump3} \ref{assump35} we have for any $t\in[T^{N,1}_{\delta}, T]$,
\begin{equation*}
u(\tilde{Z}^{N}(t))\geq\rho\frac{2\delta}{3}+\tilde{\mathcal{\zeta}}^{N,1}_{t}+\tilde{\mathcal{\zeta}}^{N,2}_{t}.
\end{equation*}
We deduce that
\begin{equation*}
\inf_{T^{N,1}_{\delta}<t<T}u(\tilde{Z}^{N}(t))\geq\rho\frac{2\delta}{3}+\inf_{T^{N,1}_{\delta}<t<T}\Big(\tilde{\mathcal{\zeta}}^{N,1}_{t}+\tilde{\mathcal{\zeta}}^{N,2}_{t}\Big).
\end{equation*}
Then 
\begin{equation}\label{limitprob}
\underset{N\to\infty}{\lim}\mathbb{P}_{z}\Big(\inf_{T^{N,1}_{\delta}<t<T}u(\tilde{Z}^{N}(t))\geq\rho\frac{\delta}{3}\Big)=1.
\end{equation}
With the convention that $\inf\emptyset=\infty$, let  $T^{N,2}_{\delta}$ the stopping time defined by
\begin{equation*}
    T^{N,2}_{\delta}:=\inf\Big\{t>T^{N,1}_{\delta}: \tilde{Z}^{N}(t)\in \tilde{\mathcal{B}}^{N}\Big\}\land T. 
\end{equation*}
As $u=0$ on the boundary $\partial O$ and $\nabla u$ is bounded, there exists a constant $\rho'$ such that if $\tilde{Z}^N(t)\in \tilde{\mathcal{B}}^N$ then $u(\tilde{Z}^N(t))<\rho'/N$ and then we deduce from \eqref{limitprob} that for any $\delta>0$,
\begin{equation*}
\underset{N\to\infty}{\lim}\mathbb{P}_{z}(T^{N,2}_{\delta}<T)=0,
\end{equation*}
and consequently
\begin{equation*}
\underset{N\to\infty}{\lim}\mathbb{P}_{z}\Big(\int_{0}^{T}\sum_{j=1}^{k}\mathbf{1}_{\big\{ \tilde{Z}^{N}(s)+\frac{h_{j}}{N}\not\in\bar{O}\big\}}ds>\delta\Big)=0.
\end{equation*}

\end{proof}

\begin{theorem}\label{LLN}
Let $\tilde{Z}^{N}(t)$ be  the sequence of processes solution of the reflecting Poisson driven $SDE$ \eqref{reflexionsde2} with the initial condition $z^{N}$. Assuming that the $\beta_{j}$'s are Lipschitz continuous, then for all $T>0$, 
\begin{equation*}
 \sup_{t\leq T}|\tilde{Z}^{N}(t)-Y(t)|\underset{N \to +\infty}{\overset{\mathbb{P}}{\longrightarrow}}0,
 \end{equation*}
 where $Y(t)$ is the unique solution of \eqref{ODE}.
\end{theorem}

\begin{proof}
We have
\begin{align*}
    \tilde{Z}^{N}(t)&= z^{N}+\frac{1}{N}\sum_{j=1}^{k}h_{j}\int_{0}^{t}\mathbf{1}_{\{\tilde{Z}^{N}(s^{-})+\frac{h_{j}}{N}\in\bar{O}\}}dQ_{s}^{N,j}   \\
    &=z^{N}+\int_{0}^{t}b(\tilde{Z}^{N}(s))ds+\widetilde{\mathcal{M}}^{N}(t)-\int_{0}^{t}\sum_{j=1}^{k}\mathbf{1}_{\big\{ \tilde{Z}^{N}(s)+\frac{h_{j}}{N}\not\in\bar{O}\big\}}\beta_{j}(\tilde{Z}^{N}(s))h_{j}ds,
\end{align*}
where $\widetilde{\mathcal{M}}^{N}(t)$ is defined as in Lemma \ref{le:1}.  We define\begin{equation*}
\Phi^{N}(t)=\widetilde{\mathcal{M}}^{N}(t)-\int_{0}^{t}\sum_{j=1}^{k}\mathbf{1}_{\big\{ \tilde{Z}^{N}(s)+\frac{h_{j}}{N}\not\in\bar{O}\big\}}\beta_{j}(\tilde{Z}^{N}(s))h_{j}ds
\end{equation*}
and
\begin{equation*}
U^{N}(t)=\tilde{Z}^{N}(t)-\Phi^{N}(t).
\end{equation*}
Then 
\begin{equation*}
|\tilde{Z}^{N}(t)-Y(t)|\leq |U^{N}(t)-Y(t)|+|\Phi^{N}(t)|.
\end{equation*}
Moreover we have the following inequalities where the second one follows from the Lipschitz character of the $\beta_{j}$'s
\begin{align*}
    |U^{N}(t)-Y(t)|&\leq |z^{N}-z| +\int_{0}^{t}|b(U^{N}(s)+\Phi^{N}(s))-b(Y(s))|ds  \\
    &\leq |z^{N}-z| +k C \int_{0}^{t}|U^{N}(s)-Y(s)|ds+k C  \int_{0}^{t}|\Phi^{N}(s)|ds .
\end{align*}
We now deduce from Gronwall's inequality 
\begin{equation*}
 |U^{N}(t)-Y(t)|\leq\Big(|z^{N}-z|+k C\int_{0}^{t}|\Phi^{N}(s)|ds\Big)\exp\{k C t\},
\end{equation*}
hence
\begin{equation*}
|\tilde{Z}^{N}(t)-Y(t)|\leq\Big(|z^{N}-z|+k C\int_{0}^{t}|\Phi^{N}(s)|ds\Big)\exp\{k C t\}+|\Phi^{N}(t)|. 
\end{equation*}
Therefore
\begin{align*}
   \sup_{0\leq t\leq T} |\tilde{Z}^{N}(t)-Y(t)|\leq  |z^{N}-z|\exp\{k C T\}+\big(1+k C T\exp\{k C T\}\big)\sup_{0\leq t\leq T}|\Phi^{N}(t)|.
\end{align*}
 Lemmas \ref{le:1} and \ref{le:2} imply that $\sup_{t\leq T}|\Phi^{N}(t)|{\overset{\mathbb{P}}{\longrightarrow}}0$, as $N\to\infty$.
 The result follows.
\end{proof}

\section{Large Deviations: preliminary results}\label{sec4}
For all $\phi\in\mathcal{AC}_{T,\bar{O}}$, the subspace of $D_{T,\bar{O}}$ consisting of absolutely continuous functions, let $\mathcal{A}_{d}(\phi)$ denote the (possibly empty) set of $\mathbb{R}^{k}_{+}$-valued Borel measurable functions $\mu$  such that
\begin{equation*}\label{allowed}
  \frac{d\phi_{t}}{dt}=\sum_{j=1}^k \mu^{j}_{t} h_{j}, \quad\text{t a.e}.
\end{equation*}
We define the rate function
\begin{equation*}
I_{T}(\phi) :=
 \begin{cases}
\underset{\mu\in\mathcal{A}_{d}(\phi)}{\inf}~~I_{T}(\phi|\mu),& \text{ if } \phi\in\mathcal{AC}_{T,\bar{O}}, \\
\infty ,& \text{ else,}
\end{cases}
\end{equation*}
where
\begin{equation*}
  I_{T}(\phi|\mu)=\int_{0}^{T}\sum_{j=1}^{k}f(\mu_{t}^{j},\beta_{j}(\phi_{t}))dt
\end{equation*}
with $f(\nu,\omega)=\nu\log(\nu/\omega)-\nu+\omega$.
 We assume in the definition of $f(\nu,\omega)$ that for all $\nu>0$, $\log(\nu/0)=\infty$ and $0\log(0/0)=0\log(0)=0$.

 The following result is a direct consequence of Lemma 4.22 in \cite{Kratz2014}
\begin{lemma}\label{semiconz}
Let $F$ a closed subset of $D_{T,A}$ and $z\in A$. We have
\begin{equation*}
\lim_{\epsilon\to0}\inf_{y\in A, |y-z|<\epsilon}\inf_{\phi\in F, \phi_{0}=y}I_{T}(\phi)=\inf_{\phi\in F, \phi_{0}=z}I_{T}(\phi).
\end{equation*}
\end{lemma}

 The next lemma states a large deviation estimate for Poisson random variables.
\begin{lemma}\label{le17}
 Let $Y_{1}$,$Y_{2}$,...be independent Poisson random variables with mean $\sigma\epsilon$. For all $N\in\mathbb{N}$, let
\begin{equation*}
 \bar{Y}^{N}=\frac{1}{N}\sum_{n=1}^{N}Y_{n}.
\end{equation*}
For any $s>0$ there exist $K, \epsilon_{0}>0$ and $N_{0}\in\mathbb{N}$ such that with
 \begin{equation}\label{g(eps)}
 g(\epsilon)=K\sqrt{\log^{-1}(\epsilon^{-1})},
 \end{equation} 
 we have
\begin{equation*}
 \mathbb{P}^{N}(\bar{Y}^{N}>g(\epsilon))<\exp\{-sN\}
\end{equation*}
for all $\epsilon<\epsilon_{0}$ and $N>N_{0}$.
\end{lemma}
\begin{proof}
 We apply Cramer's theorem, see e.g \cite{Dembo2009}, chapter 2 :
 \begin{equation*}
 \limsup_{N\to\infty}\frac{1}{N}\log(\mathbb{P}^{N}(\bar{Y}^{N}>g(\epsilon)))\leq-\inf_{x\geq g(\epsilon)}\Lambda^{*}_{\epsilon}(x),
\end{equation*}
where $\Lambda^{*}_{\epsilon}(x)=\sup_{\lambda\in\mathbb{R}}\{\lambda x-\Lambda_{\epsilon}(\lambda)\}$ with 
\begin{align*}
 &\Lambda_{\epsilon}(\lambda)=\log(\mathbb{E}(e^{\lambda Y_{1}})=\sigma\epsilon(e^{\lambda}-1).
\end{align*}
We deduce that 
\begin{equation*}
 \Lambda^{*}_{\epsilon}(x)=x\log\frac{x}{\sigma\epsilon}-x+\sigma\epsilon.
\end{equation*}
This last function is convex It reaches its infimum at $x=\sigma\epsilon$ and as $\lim_{\epsilon\rightarrow0}\frac{g(\epsilon)}{\sigma\epsilon}=+\infty$
there exists $\epsilon_{1}>0$ such that $g(\epsilon)>\sigma\epsilon$ for all $\epsilon<\epsilon_{1}$ and then, with the notation
$a_\epsilon\approx b_\epsilon$ meaning that there exists a constant $C$ such that $C^{-1} b_\epsilon\le a_\epsilon\le Cb_\epsilon$ for all $\epsilon>0$,
\begin{align*}
 \inf_{x\geq g(\epsilon)}\Lambda^{*}_{\epsilon}(x)&=g(\epsilon)\log\frac{g(\epsilon)}{\sigma\epsilon}-g(\epsilon)+\sigma\epsilon\\
 &=g(\epsilon)\log(g(\epsilon))-g(\epsilon)\log(\sigma\epsilon)-g(\epsilon)+\sigma\epsilon\\
 &\approx g(\epsilon)\log(1/\epsilon) \\
 &\approx K\sqrt{\log(1/\epsilon)}\rightarrow\infty\quad\text{as}\quad\epsilon\rightarrow0.
\end{align*}
Then there exists $\epsilon_{2}>0$ such that $\inf_{x\geq g(\epsilon)}\Lambda^{*}_{\epsilon}(x)>s$ for all $\epsilon<\epsilon_{2}$. The lemma follows by choosing $\epsilon_{0}=\min\{\epsilon_{1}, \epsilon_{2}\}$.
\end{proof}

\section{The Lower Bound} \label{Seclower}

 This section is summarized by the following result whose proof is essentially the same as that of Theorem 2.4 in section 2 of \cite{Sam2017}.
It mainly uses a Girsanov change of probability for doubly stochastic Poisson processes as well as the law of large numbers established in section \ref{weakLLN}.
\begin{theorem}\label{LDP}
Assume that both Assumption \ref{assumpinq} and Assumption \ref{assump3} are satisfied. Let $\tilde{Z}^{N,z}$ be the solution of \eqref{reflexionsde2} with $t$ restricted to $[0, T]$.
\begin{description}
  \item[a)] For $z\in \bar{O}$, $\phi\in D_{T,\bar{O}}$, $\phi_{0}=z$, $\eta>0$ and $\delta>0$ there exists $N_{\eta,\delta}\in\mathbb{N}$ such that for all $N>N_{\eta,\delta}$
  \begin{equation*}
  \inf_{y: |y-z|<\delta/2}\mathbb{P}_{y}\Big(\|\tilde{Z}^{N}-\phi\|_{T}<\delta\Big)\geq\exp\{-N(I_{T}(\phi)+\eta)\}.
  \end{equation*}
  \item[b)] For any open subset $G$ of $D_{T,\bar{O}}$, the following hold uniformly over $z\in\bar{O}$
  \begin{equation*}
 \underset{y\to z}{\liminf_{N\to\infty}}\frac{1}{N}\log\mathbb{P}_{y}(\tilde{Z}^{N}\in G)\geq-\inf_{\phi\in G, \phi_{0}=z} I_{T}(\phi).
  \end{equation*}
 \end{description}
\end{theorem}

\section{The Upper Bound} \label{Secupper}

 In \cite{Sam2017}, the upper bound was established as a consequence of a result in \cite{dupuis1991large}, which does not apply here. This is why we need to detail the proof of the upper bound.

 In this section, we shall assume that both Assumption \ref{assumpinq} and Assumption \ref{assump3} are in force.

For all $\phi\in D_{T,\bar{O}}$  and $F\subset D_{T,\bar{O}}$, we define $ \rho_{T}(\phi, F)=\inf_{\psi\in F}\|\phi-\psi\|_{T}$.
For $z\in \bar{O}$, $\delta,s>0$ we define the sets $\Phi_{z}(s)=\{\psi\in D_{T,\bar{O}}: \psi_{0}=z, I_{T}(\psi)\leq s\}$ and
$F_{\delta}^{s}(z)=\{\phi\in D_{T,\bar{O}}: \rho_{T}(\phi, \Phi_{z}(s))\geq\delta\}$.

 The following Proposition  constitutes the main step in the proof of the upper bound. 
\begin{proposition}\label{upperb}
For $\delta$, $\eta$ and $s>0$ there exists $N_{0}\in\mathds{N}$ such that
\begin{equation}\label{loexpupper}
  \mathbb{P}_{z}(\tilde{Z}^N\in F_{\delta}^{s}(z))\leq\exp\{-N(s-\eta)\}
\end{equation}
whenever $N\geq N_{0}$ and $z\in \bar{O}$.
\end{proposition}
\begin{proof}
Let $\tilde{Z}^{N}_{a}(t)=(1-a)\tilde{Z}^{N}(t)+az_{0}$ then $\|\tilde{Z}^{N}-\tilde{Z}^{N}_{a}\|_{T}<c_{1}a$ and for all $c_{1}a<\delta(d-1)/d$, 
\begin{align}\label{Inequalsup1}
 \mathbb{P}_{z}(\tilde{Z}^N\in F_{\delta}^{s}(z))
 \leq\mathbb{P}_{z}\Big(\rho_{T}(\tilde{Z}^{N}_{a},\Phi(s))\geq\frac{\delta}{d}\Big).
\end{align}
We now approximate the $\tilde{Z}^{N}$ by piecewise linear paths. Let $\epsilon>0$ be such that $T/\epsilon\in\mathds{N}$. We construct a polygonal approximation of $\tilde{Z}^{N}_{a}$ defined for all $t\in [\ell\epsilon,(\ell+1)\epsilon[$ by
\begin{equation*}
  \Upsilon_{t}=\Upsilon^{a,\epsilon}_{t}=\tilde{Z}^{N}_{a}(\ell\epsilon)\frac{(\ell+1)\epsilon-t}{\epsilon}+\tilde{Z}^{N}_{a}((\ell+1)\epsilon)\frac{t-\ell\epsilon}{\epsilon}.
\end{equation*}
Since $\{\|\tilde{Z}^{N}_{a}-\Upsilon\|_{T}<\frac{\delta}{2d}\}\cap\{\rho_{T}(\tilde{Z}^{N}_{a},\Phi_{z}(s))\geq\frac{\delta}{d}\}\subset\{\rho_{T}(\Upsilon,\Phi(s))\geq\frac{\delta}{2d}\}$,
\begin{align}\label{major2}
  \mathbb{P}_{z}\Big(\rho_{T}(\tilde{Z}^{N}_{a},\Phi_{z}(s))\geq\frac{\delta}{d}\Big)  &\leq\mathbb{P}_{z}\Big(\rho_{T}(\Upsilon,\Phi_{z}(s))\geq\frac{\delta}{2d}\Big) +\mathbb{P}_{z}\Big(\|\tilde{Z}^{N}_{a}-\Upsilon\|_{T}\geq\frac{\delta}{2d}\Big) \nonumber\\
  &\leq\mathbb{P}_{z}(I_{T}(\Upsilon)\geq s)+\mathbb{P}_{z}\Big(\|\tilde{Z}^{N}_{a}-\Upsilon\|_{T}\geq\frac{\delta}{2d}\Big).
\end{align}
We now  bound $\mathbb{P}_{z}(I_{T}(\Upsilon)\geq s)$. For any choice $\mu\in\mathcal{A}_{d}(\Upsilon)$ we have $I_{T}(\Upsilon)\leq I_{T}(\Upsilon|\mu)$ and
\begin{equation}\label{Inequalsup2}
  \mathbb{P}_{z}(I_{T}(\Upsilon)\geq s)\leq\mathbb{P}_{z}(I_{T}(\Upsilon|\mu)\geq s).
\end{equation}
Let $\{\mu^{j}_{t}, 1\leq j\leq k\}\in\mathcal{A}_{d}(\Upsilon)$ be constant on the intervals $[\ell\epsilon,(\ell+1)\epsilon[$ and equal to
\begin{equation}\label{mu1}
  \mu^{j}_{t}=\frac{1-a}{N\epsilon}\Big[P_{j}\Big(N\int_{0}^{(\ell+1)\epsilon}\beta_{j}(\tilde{Z}^{N}(s))ds\Big)-P_{j}\Big(N\int_{0}^{\ell\epsilon}\beta_{j}(\tilde{Z}^{N}(s))ds\Big)\Big].
\end{equation}
 To control the change of $\Upsilon$ over the intervals of length $\epsilon$, we will use the constant $g(\epsilon)$ from  \eqref{g(eps)} and consider the collection of events $B=\{B_{\epsilon}\}_{\epsilon>0}$ defined by
\begin{equation*}
 B_{\epsilon}=\bigcap_{\ell=0}^{T/\epsilon-1}B_{\epsilon}^{\ell},\
\text{with }
  B_{\epsilon}^{\ell}=\Big\{\sup_{\ell\epsilon\leq t_{1},t_{2}\leq (\ell+1)\epsilon}|\tilde{Z}^{N}_{i}(t_{1})-\tilde{Z}^{N}_{i}(t_{2})|\leq g(\epsilon)\quad\text{for}\quad i=1,...,d\Big\}.
\end{equation*}
We have
\begin{equation}\label{major1}
  \mathbb{P}_{z}(I_{T}(\Upsilon|\mu)>s)\leq\mathbb{P}_{z}(\{I_{T}(\Upsilon|\mu)>s\}\cap B_{\epsilon})+\mathbb{P}_{z}(B^{c}_{\epsilon}).
\end{equation}
Combining  \eqref{Inequalsup1}, \eqref{major2},  \eqref{Inequalsup2} and \eqref{major1}, we deduce that
\begin{equation}\label{Inequalsup}
 \mathbb{P}_{z}(\tilde{Z}^{N}\in F_{\delta}^{s}(z))\leq\mathbb{P}_{z}(\{I_{T}(\Upsilon|\mu)>s\}\cap B_{\epsilon})+\mathbb{P}_{z}(B^{c}_{\epsilon})+\mathbb{P}_{z}\Big(\|\tilde{Z}^{N}_{a}-\Upsilon\|_{T}\geq\frac{\delta}{2d}\Big).
\end{equation}
The next Lemmas give appropriate upper bounds for the three terms in the right side of \eqref{Inequalsup}. 
The proof of the first one relies upon Lemma \ref{le17}.
\begin{lemma}
For any $s>0$ there exists $\epsilon_{s}>0$, $N_{0}\in\mathbb{N}$ and $K>0$ such that
\begin{equation}\label{ineqsup2} 
\mathbb{P}_{z}(B_{\epsilon}^{c})+\mathbb{P}_{z}(\|\tilde{Z}^{N}_{a}-\Upsilon\|_{T}>\delta/2d)<2\frac{dkT}{\epsilon}\exp\{-sN\} 
\end{equation}
for all $\epsilon<\epsilon_{0}$, $N>N_{0}$ and any $z\in \bar O$. 
\end{lemma}

\begin{proof}
It is enough to show that the two terms on the left of \eqref{ineqsup2} have $\frac{dkT}{\epsilon}\exp\{-sN\}$ as upper bound.
For the first terms in that left side, we first remark that for all $j=1,...,k$ and $\ell=1,...,T/\epsilon$ we can write
\begin{equation*}
 \int_{0}^{(\ell+1)\epsilon}\beta_{j}(\tilde{Z}^{N}_{s})ds<\int_{0}^{\ell\epsilon}\beta_{j}(\tilde{Z}^{N}_{s})ds+\sigma\epsilon.
\end{equation*}
Moreover, we have
\begin{equation*}
 B_{\epsilon}^{c}=\bigcup_{i=1,...,d}\bigcup_{\ell=1,...,T/\epsilon}\Big\{\sup_{(\ell-1)\epsilon\leq t_{1},t_{2}\leq \ell\epsilon}|\tilde{Z}^{N}_{i}(t_{1})-\tilde{Z}^{N}_{i}(t_{2})|>g(\epsilon)\Big\}.
\end{equation*}
Thus
\begin{equation*}
 \mathbb{P}_{z}(B_{\epsilon}^{c})\leq\sum_{i=1}^{d}\sum_{\ell=1}^{T/\epsilon}\mathbb{P}_{z}\Big\{\sup_{(\ell-1)\epsilon\leq t_{1},t_{2}\leq \ell\epsilon}|\tilde{Z}^{N}_{i}(t_{1})-\tilde{Z}^{N}_{i}(t_{2})|>g(\epsilon)\Big\}.
\end{equation*}
Using \eqref{reflexionsde2} we have
\begin{align*}
& \sup_{(\ell-1)\epsilon\leq t_{1},t_{2}\leq \ell\epsilon}|\tilde{Z}^{N}_{i}(t_{1})-\tilde{Z}^{N}_{i}(t_{2})| \\
 &=\sup_{(\ell-1)\epsilon\leq t_{1},t_{2}\leq \ell\epsilon}\Big|\sum_{j}\frac{h^{i}_{j}}{N}\Big[\int_{0}^{t_{1}}\big(1-\mathbf{1}_{\{\tilde{Z}^{N}(s)+\frac{h_{j}}{N}\not\in\bar{O}\}}\big)dQ^{N,j}_{s}-\int_{0}^{t_{2}}\big(1-\mathbf{1}_{\{\tilde{Z}^{N}(s)+\frac{h_{j}}{N}\not\in\bar{O}\}}\big)dQ^{N, j}_{s}\Big]\Big|\\
 &\leq\frac{1}{N}\sum_{j}\Big|\int_{(\ell-1)\epsilon}^{\ell\epsilon}\big(1-\mathbf{1}_{\{\tilde{Z}^{N}(s)+\frac{h_{j}}{N}\not\in\bar{O}\}}\big)dQ^{N, j}_{s}\Big|\\
 &\leq\frac{1}{N}\sum_{j}\Big|P_{j}\Big(N\int_{0}^{(\ell-1)\epsilon}\beta_{j}(\tilde{Z}^{N}(s))ds+N\sigma\epsilon\Big)-P_{j}\Big(N\int_{0}^{(\ell-1)\epsilon}\beta_{j}(\tilde{Z}^{N}(s))ds\Big)\Big|\\
 &\leq\frac{1}{N}\sum_{j}Z_{j},
\end{align*}
where $Z_{j}$  $j=1,...,k$ are Poisson random variables with the mean $N\sigma\epsilon$.
Then 
\begin{equation*}
 \mathbb{P}_{z}\Big\{\sup_{(\ell-1)\epsilon\leq t_{1},t_{2}\leq \ell\epsilon}|\tilde{Z}^{N}_{i}(t_{1})-\tilde{Z}^{N}_{i}(t_{2})|>g(\epsilon)\Big\}\leq k\mathbb{P}^{N}(N^{-1}Z_{1}>g(\epsilon)/k)
\end{equation*}
Since it is a Poisson random variable with mean $N\sigma\epsilon$, $Z_1$ is the sum of $N$ iid Poisson random variable with mean $\sigma\epsilon$. Hence,
from lemma \ref{le17}, for each $s>0$ there exist  constants $K>0$, $\epsilon^1_{s}>0$ and $N_{0}\in\mathbb{N}$ such that
\begin{equation*}
 \mathbb{P}_{z}\Big\{\sup_{(\ell-1)\epsilon\leq t_{1},t_{2}\leq \ell\epsilon}|\tilde{Z}^{N}_{i}(t_{1})-\tilde{Z}^{N}_{i}(t_{2})|>g(\epsilon)\Big\}\leq k\exp\{-sN\}
\end{equation*}
for all $\epsilon<\epsilon^1_{s}$ and $N>N_{0}$. Consequently
\begin{equation*}
\mathbb{P}_{z}(B_{\epsilon}^{c})<\frac{dkT}{\epsilon}\exp\{-sN\}, 
\end{equation*}
which is the first half of \eqref{ineqsup2}. We now establish the second half.

 We first show that there exist  $\epsilon_{s}\leq\epsilon^1_{s}$ and $N_{0}\in\mathbb{N}$ such that for all $\epsilon<\epsilon_{s}$, $N>N_{0}$ and any $z\in \bar O$,
\begin{equation*}
\mathbb{P}_{z}(\|\tilde{Z}^{N}_{a}-\Upsilon\|_{T}>\delta/2d)<\frac{dkT}{\epsilon}\exp\{-sN\}.
\end{equation*}
We deduce from \eqref{reflexionsde2} that for all $t\in[\ell\epsilon,(\ell+1)\epsilon[$
\begin{align*}\label{major3}
 |\tilde{Z}^{N,a}_{i}(t)-\Upsilon^{i}_{t}|&\leq \frac{t-\ell\epsilon}{\epsilon}|\tilde{Z}_{i}^{N,a}((\ell+1)\epsilon)-\tilde{Z}_{i}^{N,a}(t)|+\frac{(\ell+1)\epsilon-t}{\epsilon}|\tilde{Z}_{i}^{N,a}(t)-\tilde{Z}_{i}^{N,a}(\ell\epsilon)| \\
 &\leq \frac{t-\ell\epsilon}{\epsilon}\sum_{j}\frac{1}{N}|\int_{t}^{(\ell+1)\epsilon}\big(1-\mathbf{1}_{\{\tilde{Z}^{N}(s)+\frac{h_{j}}{N}\not\in\bar{O}\}}\big)dQ^{N, j}_{s}| \\
 &+\frac{(\ell+1)\epsilon-t}{\epsilon}\sum_{j}\frac{1}{N}|\int_{\ell\epsilon}^{t}\big(1-\mathbf{1}_{\{\tilde{Z}^{N}(s)+\frac{h_{j}}{N}\not\in\bar{O}\}}\big)dQ^{N, j}_{s}| \\
 &\leq\sum_{j}\frac{1}{N}\Big|P_{j}\Big(N\int_{0}^{(\ell+1)\epsilon}\beta_{j}(\tilde{Z}^{N}(s))ds\Big)-P_{j}\Big(N\int_{0}^{\ell\epsilon}\beta_{j}(\tilde{Z}^{N}(s))ds\Big)\Big|\\
 &\leq\frac{1}{N}\sum_{j}\Big|P_{j}\Big(N\int_{0}^{\ell\epsilon}\beta_{j}(\tilde{Z}^{N}(s))ds+N\sigma\epsilon\Big)-P_{j}\Big(N\int_{0}^{\ell\epsilon}\beta_{j}(\tilde{Z}^{N}(s))ds\Big)\Big|\\
 &\leq\frac{1}{N}\sum_{j}Z_{j},
\end{align*}
where the $Z_{j}$ are as in the first part of the proof. Let $\epsilon^2_{0}$ be the largest $\epsilon$ such that $\delta/kd>g(\epsilon)$. Then we have from Lemma \ref{le17} that
for all $\epsilon<\epsilon_{0}=\min\{\epsilon^1_{0},\epsilon^2_{0}\}$ and $N>N_{0}$
\begin{align*}
 \mathbb{P}_{z}(\|\tilde{Z}^{N,a}-\Upsilon\|_{T}>\delta)&\leq\mathbb{P}_{z}\Big(\bigcup_{i=1}^{d}\{|\tilde{Z}^{N,a}_{i}(t)-\Upsilon^{i}_{t}|>\frac{\delta}{d}\}\quad\text{for some}\quad t\in[0,T]\Big)\\
 & \leq \frac{T}{\epsilon}\max_{0\leq \ell\leq T/\epsilon-1}\mathbb{P}_{z}\Big(\bigcup_{i=1}^{d}\{|\tilde{Z}^{N,a}_{i}(t)-\Upsilon^{i}_{t}|>\frac{\delta}{d}\}\quad\text{for some}\quad t\in[\ell\epsilon,(\ell+1)\epsilon[\Big)\\
 &\leq \frac{dkT}{\epsilon}\mathbb{P}_{z}(Z_{1}/N>\delta/kd)\leq \frac{dkT}{\epsilon} \exp\{-sN\}.
\end{align*}
The result follows.
\end{proof}

 It remains to upper bound $\mathbb{P}_{z}(\{I_{T}(\Upsilon|\mu)>s\}\cap B_{\epsilon})$ from the right hand side of \eqref{Inequalsup}. We first deduce from Chebyshev's inequality that  for all $0<\alpha<1$
\begin{equation}\label{cheby}
  \mathbb{P}_{z}(\{I_{T}(\Upsilon|\mu)>s\}\cap B_{\epsilon})\leq\frac{\mathbb{E}_{z}(\exp\{\alpha N I_{T}(\Upsilon|\mu)\}\mathfrak{1}_{B_{\epsilon}})}{\exp\{\alpha N s\}}.
\end{equation}
In order to conclude the proof of Proposition \ref{upperb}, all we need to do is to get an upper bound of the numerator in the right hand side of \eqref{cheby} of the type $\exp\{N\delta\}$, with $\delta$ arbitrarily small. This will be achieved in Lemma \ref{lemmaux1}. Note that the ideas behind this proof come from \cite{dolgoarshinnykhsample} and the proof of Theorem 3.2.2, chapter 3 in \cite{freidlin2012random}. We first establish 

\begin{lemma}\label{lemmaux}
For all $0<\alpha<1$, $j=1,...,k$ and $\ell=0,...,T/\epsilon-1$, there exists $W_{j}$ which conditionally upon $\mathcal{F}^{N}_{\ell\epsilon}$ are mutually independent Poisson random variables with respective mean $N\epsilon\beta^{j}_{\ell}=N\epsilon(\beta_{j}(\tilde{Z}^{N}(\ell\epsilon))+C d g(\epsilon))$, such that if
\begin{equation*}
 \Theta_{j}^{\ell}=\exp\Big\{\alpha N \int_{\ell\epsilon}^{(\ell+1)\epsilon}f(\mu^{j}_{t}, \beta_{j}(\Upsilon_{t}))dt\Big\}\mathfrak{1}_{B_{\epsilon}^{\ell}}
\end{equation*}
and
\begin{equation*}
\Xi_{j}^{\ell}=\exp\{\alpha N\epsilon(\sigma+2Cdg(\epsilon))\}\exp\Big\{\alpha N\epsilon f\Big(\frac{(1-a)W_{j}}{\epsilon N}, \beta^{a,j}_{\ell}\Big)\Big\}
\end{equation*}
where $\beta^{a,j}_{\ell}=(\beta_{j}(\Upsilon_{\ell\epsilon})-C d g(\epsilon))_{+}$, then

\begin{equation}\label{convex1}
  \Theta_{j}^{\ell}\leq\Xi_{j}^{\ell}\quad\text{a.s}
\end{equation}
\end{lemma}

\begin{proof}
On $B^{\ell}_{\epsilon}$, if $g(\epsilon)<1$ and $t\in[\ell\epsilon, (\ell+1)\epsilon]$, we have
\begin{equation*}
  \Big|N\int_{\ell\epsilon}^{(\ell+1)\epsilon}\beta_{j}(\tilde{Z}^{N}(t))dt-N\epsilon\beta_{j}(\tilde{Z}^{N}(\ell\epsilon))\Big|\leq N\epsilon C d g(\epsilon),\quad j=1,...,k.
\end{equation*}
If $\mu^{j}_{t}$, $j=1,...,k$ are defined by \eqref{mu1}, we have
\begin{equation}\label{twoz1z2}
  0\leq\mu^{j}_{\ell\epsilon}\leq\frac{(1-a)W_{j}}{\epsilon N}\quad\text{a.s.},
\end{equation}
where 
\begin{align*}
  W_{j} &= P_{j}\Big(N\int_{0}^{\ell\epsilon}\beta_{j}(\tilde{Z}^{N}(s))ds+\epsilon N(\beta_{j}(\tilde{Z}^{N}(\ell\epsilon))+C d g(\epsilon))\Big)-P_{j}\Big(N\int_{0}^{\ell\epsilon}\beta_{j}(\tilde{Z}^{N}(s))ds\Big).
\end{align*}
Moreover on the event $B^{\ell}_{\epsilon}$  for all $t\in[\ell\epsilon, (\ell+1)\epsilon]$,
 if $\beta^{a,j}_{\ell}=(\beta_{j}(\Upsilon_{\ell\epsilon})-C d g(\epsilon))_{+}$,
\begin{equation*}
  \beta^{a,j}_{\ell}\leq\beta_{j}(\Upsilon_{t})\leq\beta^{a,j}_{\ell}+2Cdg(\epsilon).
\end{equation*}
Hence, again on $B^{\ell}_{\epsilon}$
\begin{equation*}
  f(\mu^{j}_{t}, \beta_{j}(\Upsilon_{t}))\leq f(\mu^{j}_{t}, \beta_{\ell}^{a,j})+2C d g(\epsilon).
\end{equation*}
In fact,
\begin{align*}
  f(\mu^{j}_{t}, \beta_{j}(\Upsilon_{t}))&=\mu_{t}^{j}\log\frac{\mu_{t}^{j}}{\beta_{j}(\Upsilon_{t})}-\mu_{t}^{j}+\beta_{j}(\Upsilon_{t})\\
&\leq\mu_{t}^{j}\log\frac{\mu_{t}^{j}}{\beta_{\ell}^{a,j}}-\mu_{t}^{j}+\beta_{\ell}^{a,j}+2Cdg(\epsilon)+\mu_{t}^{j}\log\frac{\beta_{\ell}^{a,j}}{\beta_{j}(\Upsilon_{t})} \\
&\leq f(\mu^{j}_{t}, \beta_{\ell}^{a,j})+2C d g(\epsilon)\quad \text{since}\quad\log\frac{\beta_{\ell}^{a,j}}{\beta_{j}(\Upsilon_{t})}<0 .
\end{align*}

As $\mu^{j}_{t}=\mu^{j}_{\ell\epsilon}$ is constant over the interval $[\ell\epsilon, (\ell+1)\epsilon[$, we deduce that on $B^{\ell}_{\epsilon}$
\begin{equation}\label{ineqexpf}
  \exp\Big\{\alpha N \int_{\ell\epsilon}^{(\ell+1)\epsilon}f(\mu^{j}_{t}, \beta_{j}(\Upsilon_{t}))dt\Big\} \leq\exp\{\alpha N \epsilon f(\mu^{j}_{\ell\epsilon},\beta_{\ell}^{a,j})+2\alpha N C d \epsilon g(\epsilon)\}.
\end{equation}
\eqref{convex1} follows  from \eqref{twoz1z2}, \eqref{ineqexpf} and the convexity of $f(\nu,\omega)$ in $\nu$.
\end{proof}

 The next Proposition gives us an upper bound for the conditional expectation of the right hand side of the inequality \eqref{convex1}.
\begin{proposition}\label{propaux}
Let $a=h(\epsilon)=\Big[-\log g^{1/2}(\epsilon)\Big]^{-\frac{1}{\nu}}$ where $\nu$ is the constant in the Assumption \ref{assumpinq} \ref{assump25}. For all $0<\alpha<1$ there exist $\epsilon_{\alpha}$, $K_{\alpha}$ and $\tilde{K}$ such that for all $\epsilon\leq\epsilon_{\alpha}$ we have, with $g(\epsilon)$ defined by \eqref{g(eps)}, for any $z\in \bar O$,
\begin{align*}
  \mathbb{E}_{z}\Big(\exp\Big\{\alpha N\epsilon f\Big(\frac{(1-a)W_{j}}{\epsilon N}, \beta^{a,j}_{\ell}\Big)\Big\}|\mathcal{F}^{N}_{\ell\epsilon}\Big)\Big\} 
  \leq K_{\alpha}\exp\{N\epsilon\tilde{K}(1-\alpha+2h(\epsilon)+2dg(\epsilon))\}.
\end{align*}

\end{proposition}

\begin{proof}
 Conditionally upon $\mathcal{F}^{N}_{\ell\epsilon}$, $W_{j}$ is a Poisson random variable with mean $N\epsilon\beta_{\ell}^{j}$. Moreover we have, see the definitions of $\beta^{a,j}_{\ell}$ and $\beta^{j}_{\ell}$ in the statement of Lemma \ref{lemmaux},
\begin{equation*}
   |\beta_{\ell}^{a,j}-\beta_{\ell}^{j}|\leq \tilde{C}(a+2d g(\epsilon)).
\end{equation*}
With $\tilde{\epsilon}=\epsilon/(1-a)$ and $\tilde{\alpha}=(1-a)\alpha$, we have
\begin{align}\label{Inequal1}
 & \mathbb{E}_{z}\Big(\exp\Big\{\alpha N\epsilon f\Big(\frac{(1-a)W_{j}}{\epsilon N}, \beta^{a,j}_{\ell}\Big)\Big\}|\mathcal{F}^{N}_{\ell\epsilon}\Big)=\mathbb{E}_{z}\Big(\exp\Big\{\alpha N\epsilon f\Big(\frac{W_{j}}{\tilde{\epsilon} N}, \beta^{a,j}_{\ell}\Big)\Big\}|\mathcal{F}^{N}_{\ell\epsilon}\Big) \nonumber\\
  &=\sum_{m\geq0}\exp\Big\{\alpha N\epsilon f\Big(\frac{m}{\tilde{\epsilon} N}, \beta^{a,j}_{\ell}\Big)\Big\}\frac{(N\epsilon\beta_{\ell}^{j})^{m}\exp\{-N\epsilon\beta_{\ell}^{j}\}}{m!} \nonumber\\
  &=\sum_{m\geq0}\exp\Big\{\alpha N\epsilon \Big(\frac{m}{\tilde{\epsilon} N}\log\Big(\frac{m}{\tilde{\epsilon} N\beta^{a,j}_{\ell}}\Big)-\frac{m}{\tilde{\epsilon} N} +\beta^{a,j}_{\ell} \Big)\Big\}\frac{(N\epsilon\beta_{\ell}^{j})^{m}\exp\{-N\epsilon\beta_{\ell}^{j}\}}{m!} \nonumber\\
  &\leq\exp\{N\epsilon\tilde{C}(a+2d g(\epsilon))\}\sum_{m\geq0}\frac{m^{\tilde{\alpha}m}\exp\{-\tilde{\alpha}m\}}{m!} (N\epsilon\beta^{a,j}_{\ell})^{m(1-\tilde{\alpha})}\Big(\frac{\beta^{j}_{\ell}}{\beta^{a,j}_{\ell}}\Big)^{m} \exp\{-N\epsilon\beta^{a,j}_{\ell}(1-\alpha)\} \nonumber\\
  &\leq\exp\{N\epsilon C_{1}(a+2d g(\epsilon))\}\sum_{m\geq0}\frac{m^{\tilde{\alpha}m}\exp\{-\tilde{\alpha}m\}}{m!} (N\epsilon\beta^{a,j}_{\ell})^{m(1-\tilde{\alpha})}\Big(\frac{\beta^{j}_{\ell}}{\beta^{a,j}_{\ell}}\Big)^{m} \exp\{-N\epsilon\beta^{a,j}_{\ell}(1-\tilde{\alpha})\}.
\end{align}

 Since the function $v(x)=x^{m(1-\tilde{\alpha})}\exp\{-2x(1-\tilde{\alpha})\}$ reaches its maximum at $x=m/2$, 
\begin{equation*}
  (N\epsilon\beta^{a,j}_{\ell})^{m(1-\tilde{\alpha})} \exp\{-2N\epsilon\beta^{a,j}_{\ell}(1-\tilde{\alpha})\}\leq\Big(\frac{m}{2}\Big)^{m(1-\tilde{\alpha})}\exp\{-m(1-\tilde{\alpha})\}.
\end{equation*}
Thus
\begin{align}\label{Inequal2}
&\sum_{m\geq0}\frac{m^{\tilde{\alpha}m}\exp\{-\tilde{\alpha}m\}}{m!} (N\epsilon\beta^{a,j}_{\ell})^{m(1-\tilde{\alpha})}\Big(\frac{\beta^{j}_{\ell}}{\beta^{a,j}_{\ell}}\Big)^{m} \exp\{-N\epsilon\beta^{a,j}_{\ell}(1-\tilde{\alpha})\} \nonumber \\
&\leq \exp\{N\epsilon\beta^{a,j}_{\ell}(1-\tilde{\alpha})\}\sum_{m\geq0}\frac{m^{m} \exp\{-m\}}{m!}\Big(\frac{\beta^{j}_{\ell}/\beta^{a,j}_{\ell}}{2^{(1-\tilde{\alpha})}}\Big)^{m}
\end{align}
We shall show (see the proof of Lemma \ref{Aux} below)  that
for all $0<\alpha<1$ there exists $\epsilon_{\alpha}>0$ such that for all $\epsilon<\epsilon_{\alpha}$ ,
\begin{equation*}
\frac{\beta_{\ell}^{j}}{\beta_{\ell}^{a,j}}<2^{(1-\alpha)/2}<2^{(1-\tilde{\alpha})/2}.
\end{equation*}
Then for $\epsilon$ small enough we have
\begin{equation}\label{Inequal3}
\exp\{N\epsilon\beta_{\ell}^{a,j}(1-\tilde{\alpha})\}\sum_{m\geq0}\frac{m^{m}e^{-m}}{m!}\Big(\frac{\beta_{\ell}^{j,q}/\beta_{\ell}^{a,j}}{2^{(1-\tilde{\alpha})}}\Big)^{m}  
\leq e^{N\epsilon\theta(1-\tilde{\alpha})} K_{\alpha}
\end{equation}
since the above series converges. The proposition follows from \eqref{Inequal1}, \eqref{Inequal2} and \eqref{Inequal3}. 
\end{proof}

\begin{lemma}\label{Aux}
For all $0<\alpha<1$ there exists $\epsilon_{\alpha}>0$ such that for all $\epsilon<\epsilon_{\alpha}$ 
\begin{equation*}
\frac{\beta_{\ell}^{j}}{\beta_{\ell}^{a,j}}<2^{(1-\alpha)/2}<2^{(1-\tilde{\alpha})/2},
\end{equation*}
with again $\tilde{\alpha}=(1-a)\alpha$, a as in the statement of Proposition \ref{propaux}.
\end{lemma}

\begin{proof}
We have
\begin{equation*}
 \frac{\beta^{j}_{\ell}}{\beta^{a,j}_{\ell}}\leq\frac{\beta_{j}(\tilde{Z}^{N}(\ell\epsilon))+Cdg(\epsilon)}{(\beta_{j}(\tilde{Z}^{N,a}(\ell\epsilon))-Cdg(\epsilon))^+}
 \end{equation*}
If $\beta_{j}(\tilde{Z}^{N}(\ell\epsilon))<\lambda_{1}$ we have using the Assumptions \ref{assumpinq} \ref{assump24} and  \ref{assump25},
\begin{align*}
 \frac{\beta^{j}_{\ell}}{\beta^{a,j}_{\ell}}&\leq\frac{\beta_{j}(\tilde{Z}^{N,a}(\ell\epsilon))+Cdg(\epsilon)}{(\beta_{j}(\tilde{Z}^{N,a}(\ell\epsilon))-Cdg(\epsilon))^+} \\
 &\leq\frac{C_{a}+Cdg(\epsilon)}{(C_{a}-Cdg(\epsilon))^+}\leq\frac{1+\frac{Cdg(\epsilon)}{g^{1/2}(\epsilon)}}{\Big(1-\frac{Cdg(\epsilon)}{g^{1/2}(\epsilon)}\Big)^+}\rightarrow1\quad{as}\quad\epsilon\rightarrow0.
 \end{align*}
If $\beta_{j}(\tilde{Z}^{N}(\ell\epsilon))\geq\lambda_{1}$, we have
\begin{align*}
 &\frac{\beta^{j}_{\ell}}{\beta^{a,j}_{\ell}}\leq\frac{\beta_{j}(\tilde{Z}^{N}(\ell\epsilon))+Cdg(\epsilon)}{(\beta_{j}(\tilde{Z}^{N}(\ell\epsilon))-C\bar{C}h(\epsilon)-Cdg(\epsilon))^+}\\
 &\leq\frac{\lambda_{1}+Cdg(\epsilon)}{(\lambda_{1}-C\bar{C}h(\epsilon)-Cdg(\epsilon))^+}\rightarrow1\quad\text{as}\quad\epsilon\rightarrow0.
 \end{align*}
Then there exists $\epsilon_{\alpha}$ such that $\frac{\beta_{\ell}^{j}}{\beta_{\ell}^{a,j}}<2^{(1-\alpha)/2}<2^{(1-\tilde{\alpha})/2}$ for all $\epsilon<\epsilon_{\alpha}$.
\end{proof}

 The next lemma gives us an upper bound for the quantity $\mathbb{E}_{z}\Big(\exp\{\alpha NI_{T}(\Upsilon|\mu)\}\mathbf{1}_{B_{\epsilon}}\Big)$.
\begin{lemma}\label{lemmaux1}
For all $0<\alpha<1$ there exist $\epsilon_{\alpha}$, $K_{\alpha}$ and $\tilde{K}_{1}$ such that for all $\epsilon\leq\epsilon_{\alpha}$, we have the following inequality
\begin{equation}\label{ineqsup1}
\mathbb{E}_{z}\Big(\exp\{\alpha N I_{T}(\Upsilon|\mu)\}\mathbf{1}_{B_{\epsilon}}\Big)\leq (2K_{\alpha})^{\frac{kT}{\epsilon}}\exp\{kNT\tilde{K}_{1}(1-\alpha+h(\epsilon)+4dg(\epsilon))\},
\end{equation}
for any $z\in \bar O$.
\end{lemma}
\begin{proof}
We first deduce from Lemma \ref{lemmaux} and Proposition \ref{propaux} 
\begin{equation*}
 \mathbb{E}_{z}(\Theta_{j}^{\ell}|\mathcal{F}^{N}_{\ell\epsilon})\leq\mathbb{E}_{z}(\Xi_{j}^{\ell}|\mathcal{F}^{N}_{\ell\epsilon})\leq K_{\alpha}\exp\{N\epsilon\tilde{K}_{1}(1-\alpha+2h(\epsilon)+4dg(\epsilon))\}.
\end{equation*}

 Moreover, the $\Xi^{\ell}_{j}$, $j=1,...,k$ are conditionnally independent given $\mathcal{F}^{N}_{\ell\epsilon}$. So we can take successively the conditional expectations with 
respect to  $\mathcal{F}^{N}_{(\frac{T}{\epsilon}-1)\epsilon}$, $\mathcal{F}^{N}_{(\frac{T}{\epsilon}-2)\epsilon}$,...,$\mathcal{F}^{N}_{\epsilon}$, to conclude that for all $0<\alpha<1$ and $\epsilon<\epsilon_{\alpha}$,
\begin{align*}
\mathbb{E}_{z}\Big(\exp\{\alpha NI_{T}(\Upsilon|\mu)\}\mathbf{1}_{B_{\epsilon}}\Big)
&=\mathbb{E}_{z}\Big(\mathbb{E}_{z}\Big(\prod_{\ell=0}^{\frac{T}{\epsilon}-1}\prod_{j=1}^{k}\Theta_{j}^{\ell}|\mathcal{F}^{N}_{(\frac{T}{\epsilon}-1)\epsilon}\Big)\Big)\leq\mathbb{E}_{z}\Big(\mathbb{E}_{z}\Big(\prod_{\ell=0}^{\frac{T}{\epsilon}-1}\prod_{j=1}^{k}\Xi_{j}^{\ell}|\mathcal{F}^{N}_{(\frac{T}{\epsilon}-1)\epsilon}\Big)\Big)\\
&\leq\mathbb{E}_{z}\Big(\prod_{\ell=0}^{\frac{T}{\epsilon}-2}\prod_{j=1}^{k}\Xi_{j}^{\ell}\mathbb{E}_{z}\Big(\prod_{j=1}^{k}\Xi_{j}^{\frac{T}{\epsilon}-1}|\mathcal{F}^{N}_{(\frac{T}{\epsilon}-1)\epsilon}\Big)\Big)\\
&=(K_{\alpha})^{\frac{kT}{\epsilon}}\exp\{kNT\tilde{K}_{1}(1-\alpha+h(\epsilon)+4dg(\epsilon))\}.
\end{align*}
The result follows  \end{proof}

 We now conclude the proof of Proposition \ref{upperb}.
The upper bound of the first term in the right side of \eqref{Inequalsup} is obtained by combining \eqref{cheby} and \eqref{ineqsup1}. Indeed, for all $0<\alpha<1$, $\epsilon<\min\{\epsilon_{0},\epsilon_{\alpha},\epsilon_{1}\}$ and $a=h(\epsilon)=\Big[-\log g^{1/2}(\epsilon)\Big]^{-\frac{1}{\nu}}$,
\begin{equation}\label{Inequalsup3}
\mathbb{P}_{z}(\{I_{T}(\Upsilon|\mu)>s\}\cap B_{\epsilon})\leq (K_{\alpha})^{\frac{kT}{\epsilon}}\exp\{kNT\tilde{K}_{1}(1-\alpha+h(\epsilon)+4dg(\epsilon))\}
 \times\exp\{-\alpha N s\}
\end{equation}

 Combining \eqref{Inequalsup}, \eqref{ineqsup2} and \eqref{Inequalsup3},
 we have for all $\delta>0$, $\alpha$, $\epsilon$ and $a$ as above,
\begin{equation*}
 \mathbb{P}_{z}(\tilde{Z}^N\in F_{\delta}^{s}(z)) \leq (K_{\alpha})^{\frac{kT}{\epsilon}}\exp\{kNT\tilde{K}_{1}(1-\alpha+h(\epsilon)+4dg(\epsilon))\}\times\exp\{-\alpha N s\}+\frac{2dTk}{\epsilon}\exp\{-sN\}.
\end{equation*}
Finally, we choose $1-\alpha$ and $\epsilon$ small enough to ensure that $kT\tilde{K}_{1}(1-\alpha+h(\epsilon)+4dg(\epsilon))<\eta/4$
and $(1-\alpha)s<\eta/4$. We also take $N$ large enough so that $kT\log(K_{\alpha})/N\epsilon<\eta/4$ and $\log(2dkT/\epsilon)/N<\eta/4$, hence we deduce that for any $z\in A$,
\begin{equation*}
 \mathbb{P}_{z}(\tilde{Z}^N\in F_{\delta}^{s}(z))\leq\exp\{-N(s-\eta)\}.
\end{equation*}
\end{proof}

 We now establish the upper bound.
\begin{theorem}\label{upperbound11}
For any closed subset $F$ of $D_{T,A}$, $z\in A$
\begin{equation}\label{upper}
\underset{y\to z}{\limsup_{N\to\infty}}\frac{1}{N}\log\mathbb{P}_{y}(\tilde{Z}^N\in F)\leq -\inf_{\phi\in F,\phi_{0}=z}I_{T}(\phi).
\end{equation}
\end{theorem}
\begin{proof}
We first assume that $\inf_{\phi\in F,\phi_{0}=z}I_{T}(\phi)<\infty$, and let $\gamma>0$ be arbitrary. By Lemma \ref{semiconz}, 
there exists $\epsilon^{\gamma}>0$ such that for all $\epsilon<\epsilon^{\gamma}$,
\begin{equation*}
y\in A, |y-z|\leq \epsilon\implies\inf_{\phi\in F, \phi_{0}=y}I_{T}(\phi)\geq\inf_{\phi\in F,\phi_{0}=z}I_{T}(\phi)-\gamma.
\end{equation*}
For $\epsilon<\epsilon^{\gamma}$, let $s=\inf_{\phi\in F, \phi_{0}=z}I_{T}(\phi)-\gamma$,
\begin{equation*}
W(\epsilon)=\{\phi\in F: |\phi_{0}-z|\leq\epsilon\}\quad \text{and} \quad
U(\epsilon)=\bigcup_{y\in A, |y-z|\leq\epsilon}\Phi_{y}(s).
\end{equation*}
 $W(\epsilon)$ is closed in $D_{T,A}$ and does not intersect the set $U(\epsilon)$, which is compact, see Proposition 4.21 in \cite{Kratz2014}. By the Hahn-Banach theorem,
\begin{equation*}
\delta^{\epsilon}=\inf_{\phi\in W(\epsilon)}\inf_{\psi\in U(\epsilon)}\|\phi-\psi\|_{T}>0.
\end{equation*}
We deduce that for all $\eta>0$ and any $y\in A$ with $|y-z|\leq\epsilon$, there exists $N_{0}\in\mathbb{N}$ such that for all $N>N_{0}$, using \eqref{loexpupper} for the second inequality,
\begin{align*}
\mathbb{P}_{y}(\tilde{Z}^N\in F)&=\mathbb{P}_{y}(\tilde{Z}^N\in W(\epsilon)) \\
&\leq\mathbb{P}_{y}(\tilde{Z}^N\in F_{\delta^{\epsilon}}^{s}(y))\\
&\leq \exp\{-N(s-\eta)\}.
\end{align*}
Consequently, for any $N>N_{0}$ (where $N_{0}$ depends upon $\epsilon$, $\gamma$ and $\eta$) and $|y-z|<\epsilon^{\gamma}$,
\begin{equation*}
\frac{1}{N}\log\mathbb{P}_{y}(\tilde{Z}^N\in F)\leq -\inf_{\phi\in F,\phi_{0}=z}I_{T}(\phi)+\gamma+\eta.
\end{equation*}
We deduce that
\begin{equation*}
\underset{y\to z}{\limsup_{N\to\infty}}\frac{1}{N}\log\mathbb{P}_{y}(\tilde{Z}^N\in F)\leq -\inf_{\phi\in F,\phi_{0}=z}I_{T}(\phi)+\gamma+\eta,
\end{equation*}
for any $\nu, \eta>0$, hence the result in the case $\inf_{\phi\in F,\phi_{0}=z}I_{T}(\phi)<\infty$. Otherwise, for any $s>0$, there exists $\epsilon>0$ such that the distance between $W(\epsilon)$ and $U(\epsilon)$ is greater that some $\delta^\epsilon>0$. Following the above argument, we deduce that
\begin{equation*}
\underset{y\to z}{\limsup_{N\to\infty}}\frac{1}{N}\log\mathbb{P}_{y}(\tilde{Z}^N\in F)\leq -s+\eta,
\end{equation*}
for any $s,\eta>0$, from which the result follows.
\end{proof}

 We deduce as in \cite{Dembo2009} Corollary 5.6.15,
\begin{corollary}
For any open subset $F$ of $D_{T,A}$ and any compact subset $K$ of $A$,
  \begin{equation*}
\limsup_{N\to\infty}\frac{1}{N}\log\sup_{z\in K}\mathbb{P}_{z}(\tilde{Z}^N\in F)\leq-\inf_{z\in K}\inf_{\phi\in F, \phi_{0}=z}I_{T}(\phi).
\end{equation*}
\end{corollary}

\section{Applications}\label{sec7}
 In this section we present two models of infectious diseases for which the hypothesis \ref{assump3} is verified. Indeed, the characteristic boundary (that we note $\widetilde{\partial O}$) between the basin of attraction of the disease-free equilibrium $\bar z$ and the basin of attraction of the endemic equilibrium $z^{\ast}$ is the global stable manifold of the saddle point $\tilde z$ (another point of endemic equilibrium). We know from \cite{guckenheimer2013nonlinear} (Theorem 1.3.1, page 13) that this boundary is a curve of class $C^{\infty}$. 

\subsection{A model with vaccination $(SIV)$}

\subsubsection{Description of the model} \label{example1}
\begin{figure} 
\centering
\begin{picture} (290,140)
\linethickness{0.4mm}

\put(135,0){\line(1,0){40}}
\put(135,0){\line(0,1){40}}
\put(175,0){\line(0,1){40}}
\put(135,40){\line(1,0){40}}

\put(60,90){\line(1,0){40}}
 \put(60,90){\line(0,1){40}}
\put(100,90){\line(0,1){40}}
\put(60,130){\line(1,0){40}}

\put(200,90){\line(1,0){40}}
\put(200,90){\line(0,1){40}}
\put(240,90){\line(0,1){40}}
\put(200,130){\line(1,0){40}}

\put(190,95){\vector(-1,0){80}}
\put(110,125){\vector(1,0){80}}
\put(70,85){\vector(3,-4){55}}
\put(180,10){\vector(3,4){55}}
\put(130,40){\vector(-3,4){30}}

\put(250,110){\vector(1,0){40}}
\put(45,95){\vector(-1,0){40}}
\put(5,125){\vector(1,0){40}}
\put(180,5){\vector(1,0){40}}

\put(72,100){\huge{$S$}}
\put(215,100){\huge{$I$}}
\put(144,10){\huge{$V$}}

\put(120,129){\small{$\beta S^N I^N/N$}}
\put(120,115){\small{(infection rate)}}

\put(120,99){\small{$\gamma I^N$}}
\put(120,85){\small{(recovery rate)}}

\put(65,50){\small{$\eta S^N$}}
\put(20,40){\small{(vaccination rate})}

\put(115,65){\small{$\theta V^N$}}
\put(123,55){\small{(loss of vaccination)}}

\put(15,129){\small{$\mu N$}}
\put(5,115){\small{(birth rate)}}

\put(15,99){\small{$\mu S^N$}}
\put(5,85){\small{(death rate)}}

\put(265,114){\small{$\mu I^N$}}
\put(220,50){\small{$\chi \beta V^N I^N/N \,\, (\chi \in [0,1])$}}

\put(215,40){\small{(infection of vaccinated)}}
\put(200,10){\small{$\mu V^N$}}

\end{picture}
\caption{Diagram of the $S^NI^NV^N$ model.}
\label{figSIV}
\end{figure}
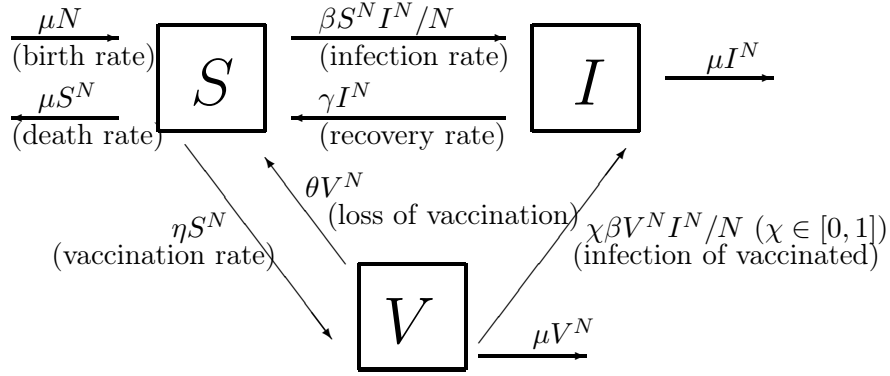

 We consider the so-called $S^NI^NV^N$ model: it is a model with vaccination and demography. 
In this model, we suppose that a population with constant size $N$ is  divided into three compartments namely:  $S^N(t)$, $I^N(t)$ and $V^N(t)$ are respectively the number of susceptible, infectious  and vaccinated individuals at time $t$. Figure~\ref{figSIV} gives us a good graphical representation of the disease transmission for this model. We assume that susceptibles are vaccinated at rate $\eta$ and lose their protection at rate $\theta$; the vaccine is not perfect but decreases the rate of infection by a factor $\chi\in[0,1]$. While each infected infects a given susceptible at rate $\beta S^{N}(t)/N$, it infects a given vaccinated individual at rate $\chi\beta V^{N}/N$, where $\beta=r\kappa$, $\kappa$ being the mean number of individuals met by one infected per unit time, $r$ (resp $r\chi$) is the probability that an encounter between an infected and a susceptible (resp a vaccinated) results in an infection. 

 The proportions of infectious $i^N(t)=\frac{I^N(t)}{N}$ and vaccinated $v^N(t)=\frac{V^N(t)}{N}$ can be described as follows. Let $P_{1},P_{2},P_{3},P_{4}$ $,P_{5},P_{6},P_{7}$ be iid standard Poisson processes:
\begin{align*}
   i^{N}(t)&=i^{N}(0)+\frac{1}{N}P_1\Big(N\beta\int_0^t i^{N}(s)(1-i^{N}(s)-v^N(s))ds\Big)+\frac{1}{N}P_2\Big(N\chi\beta\int_0^t i^{N}(s)v^N(s)ds\Big) \nonumber\\
   &-\frac{1}{N}P_3\Big(N\gamma\int_0^t i^{N}(s)ds\Big)-\frac{1}{N}P_6\Big(N\mu\int_0^t i^{N}(s)ds\Big) ,\\
    v^N(t)&=v^N(0)-\frac{1}{N}P_2\Big(N\chi\beta\int_0^t i^{N}(s)v^N(s)ds\Big)-\frac{1}{N}P_4\Big(N\theta\int_0^t i^{N}(s)v^N(s)ds\Big) \nonumber\\
    &+\frac{1}{N}P_5\Big(N\eta\int_0^t (1-i^{N}(s)-v^N(s))ds\Big)-\frac{1}{N}P_7\Big(N\mu\int_0^t v^{N}(s)ds\Big) \nonumber.
\end{align*}
If we let 
\begin{align*}
&h_{1}= (1,0)^\top, \quad \beta_{1}(z)=\beta z_{1} (1-z_{1}-z_{2}), \\
&h_{2}= (1,-1)^\top, \quad \beta_{2}(z)=\chi\beta z_{1} z_{2}, \\
&h_{3}= (-1,0)^\top, \quad \beta_{3}(z)= \gamma z_1, \\
&h_{4}= (0,-1)^\top, \quad \beta_{4}(z)=\theta z_{2}, \\
&h_{5}= (0,1)^\top, \quad \beta_{5}(z)=\eta (1-z_{1}- z_{2})\\ 
&h_{6}= (-1,0)^\top, \quad \beta_{6}(z)= \mu z_{1}, \\
&h_{7}= (0,-1)^\top, \quad \beta_{7}(z)= \mu z_{2},
\end{align*}
 our epidemic model takes the form
\begin{equation}\label{PoissonSIV}
Z^{N,z_0}(t)=\frac{\left[N z_0\right]}{N}+\frac{1}{N}\sum_{j=1}^{7}h_{j}P_{j}\Big(N\int_{0}^{t}\beta_{j}(Z^{N,z_0}(s))ds\Big),
\end{equation}
where  $Z^{N,z_0}(t)=(i^N(t), v^N(t))$.

 It is easy to show that for all $T>0$ the process $(Z^{N,z_0}(t))_{0\le t\le T}$ defined by \eqref{PoissonSIV} converges  almost surely and uniformly on $[0,T]$ as $N$ tends to $\infty$ to the solution $(i(t), v(t))$ of the following $ODE$ 
\begin{equation}\label{ODESIV}
\begin{cases}
      \frac{dz_1}{dt}(t)=(\beta-\mu-\gamma) z_1(t)-\beta(1-\chi) z_1(t) z_2(t)-\beta z_1^{2}(t) \\
      \frac{dz_2}{dt}(t)=\eta-\eta z_1(t)-(\eta+\mu+\theta) z_2(t)-\chi\beta z_1(t) z_2(t).
\end{cases}
\end{equation}
where $i(t)$, $v(t)$ denote respectively the proportion of infectious and vaccinated at time $t$. This deterministic $SIV$ model \eqref{ODESIV} is the two dimensional version of the $SIV$ model studied in \cite{Kribs2000}  by Kribs-Zaleta and Velasco-Hern\'andez (see Theorem 1). They show that if $(\mu+\theta+\chi \eta)^{2}<(\mu+\gamma)\chi(1-\chi)\eta$ and $\beta_{1}<\beta<\beta_{0}$  where
\begin{align*}
  \beta_{0}&=(\mu+\gamma)\frac{\mu+\theta+\eta}{\mu+\theta+\chi\eta}  \\
    \beta_{1}&=\mu+\gamma-\frac{\mu+\theta+\chi\eta}{\sigma}+\frac{2}{\chi}\sqrt{(\mu+\gamma)\sigma(1-\chi)\eta}, 
\end{align*}
then two endemic equilibria $z^{\ast}=(z^{\ast}_{1},z^{\ast}_{2})$, $\tilde{z}=(\tilde{z}_{1},\tilde{z}_{2})$ exist, one of which namely $z^{\ast}$ is locally stable while $\tilde{z}$ is unstable. These two equilibria are completed with the disease free equilibrium $\bar{z}$ $\big(\bar {z}_{1}=0,\bar {z}_{2}=\frac{\eta}{\mu+\theta+\eta}\big)$ which is locally stable. The figure \ref{FigSIV} shows the basin of attraction $O$ of the endemic equilibrium $z^{\ast}$. It is delimited by the boundary $\widetilde{\partial O}$ and it contains the point $z^{\ast}$. 
The first components of the equilibria $z^{\ast}$ and $\tilde{z}$ are the solutions of the equation \ref{coordendemic1} and the second components is given by equation \ref{coordendemic2} below
\begin{equation}\label{coordendemic1}
\begin{cases}
      D_{1}x^{2}+D_{2}x+ D_{3}=0 \\
      \text{where}\quad D_{1}=-\beta\chi, D_{2}=\chi(\beta-\mu-\gamma)-(\mu+\theta+\chi\eta) \\
      \text{and}\quad D_{3}=(\mu+\theta+\eta)(1-\frac{\mu+\gamma}{\beta})-(1-\chi)\eta.
\end{cases}
\end{equation}

\begin{equation}\label{coordendemic2}
z_{2}=\frac{\eta(1-z_{1})}{\mu+\theta+\eta+\beta\chi z_{1}}.
\end{equation}

\subsubsection{The boundary $\widetilde{\partial O}$ in the $SIV$ model}
The aim of this section is to establish that the assumptions \ref{assumpinq} \ref{assump21} and \ref{assumpinq} \ref{assump22} are satisfied by the model with vaccination of  \cite{Kribs2000} and with $O$ the basin of attraction of the equilibrium $z^\ast$.

We define
\begin{itemize}
  \item $(\mathcal{D}_{1.1})$ the straight line whose equation reads $(\beta-\mu-\gamma)-\beta(1-\chi) z_{2}-\beta z_{1}=0$ i.e
  \begin{equation*}
  z_{2}=-\frac{\beta}{\beta(1-\chi)}z_{1}+\frac{\beta-\mu-\gamma}{\beta(1-\chi)},
\end{equation*} 
\item $(\mathcal{H}_{1})$ the curve having the equation $\eta-\eta z_{1}-(\eta+\mu+\theta) z_{2}-\chi\beta z_{1}z_{2}=0$ i.e
\begin{equation*}
 z_{2}=\frac{\eta-\eta z_{1}}{\chi\beta z_{1}+(\eta+\mu+\theta)},
\end{equation*}
\end{itemize}

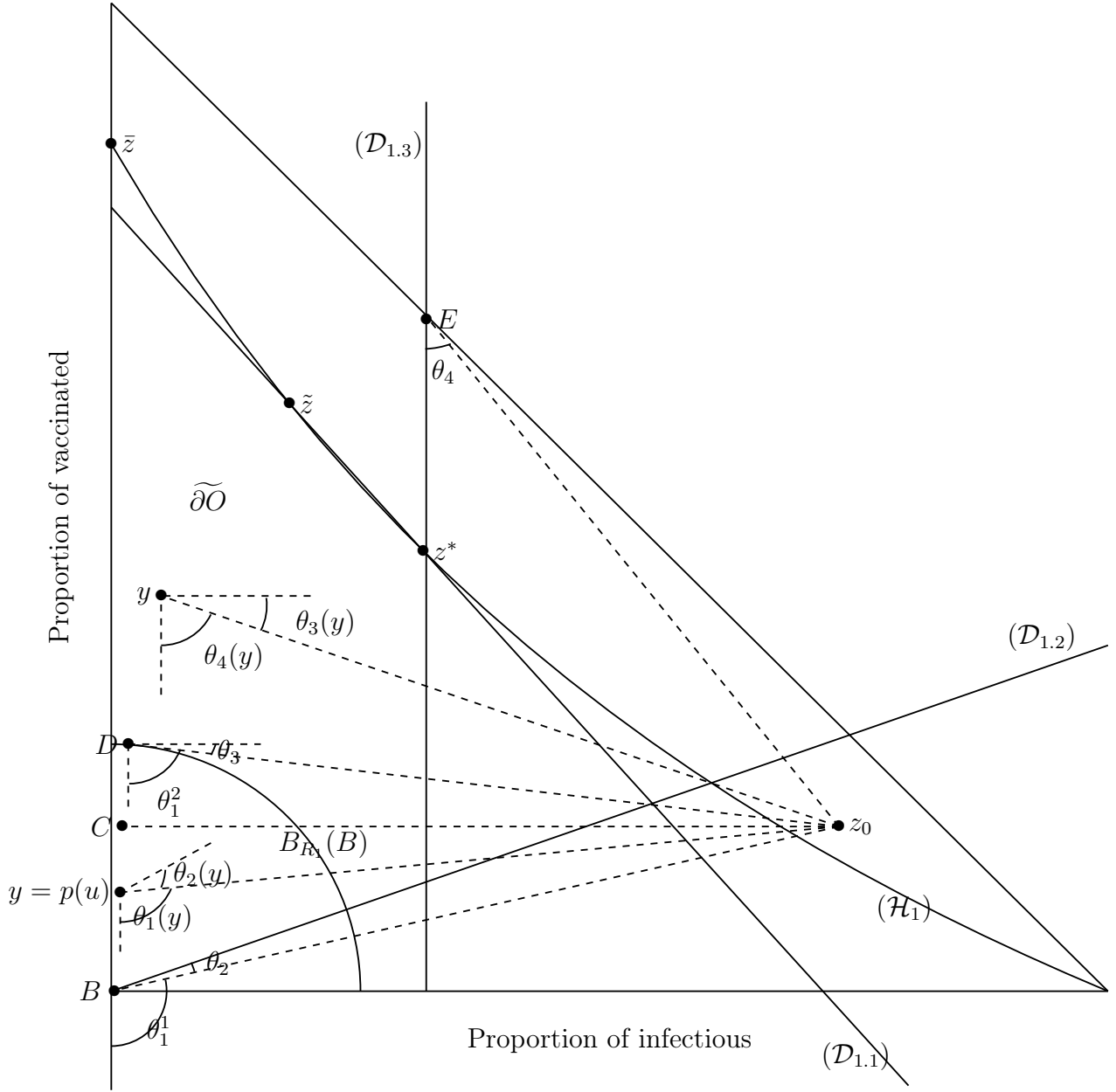
\begin{figure}
\begin{center}
\begin{tikzpicture}[scale=15, line width=0.7pt]
\draw (0,0) -- (1,0);
\draw (0,-0.1) -- (0,1);
\draw (1,0) -- (0,1);
\draw[draw=red, line width=1.1pt] plot file {DataSIV.txt};
\draw (0,0.8571428) node[right] {$\bar{z}$};
\draw (0,0.8571428) node {$\bullet$};
\draw (0.17880587,0.59453669) node[right] {$\tilde{z}$};
\draw (0.17880587,0.59453669) node {$\bullet$};
\draw (0.312861035,0.44558654) node[right] {$z^{\ast}$};
\draw (0.312861035,0.44558654) node {$\bullet$};

\draw (0.5,-0.05) node{Proportion of infectious};
\draw (-0.05,0.5) node[rotate=90]{Proportion of vaccinated};

\draw (6:0.8) node {$(\mathcal{H}_{1})$};
\draw [domain=0:1] plot(\x,{(0.3-0.3*\x)/(0.36*\x+0.35)});
\draw (355:0.75) node {$(\mathcal{D}_{1.1})$};
\draw [domain=0:0.8] plot(\x,{(2.57-3.6*\x)/3.24});

\draw (21:1) node {$(\mathcal{D}_{1.2})$};
\draw (0,0) -- (1,0.35);
\draw (72:0.9) node {$(\mathcal{D}_{1.3})$};
\draw (0.316,0) -- (0.316,0.9);

\draw (0.73,1/6) node[right] {$z_{0}$};
\draw (0.73,1/6) node {$\bullet$};
\draw[every to/.style={append after command={[draw,dashed]}}] (0.73,1/6) to (0.316,0.683);
\draw[every to/.style={append after command={[draw,dashed]}}] (0.73,1/6) to (0,0);

\draw (79:0.51) node {$\widetilde{\partial O}$};

\draw (0,-1/18) arc (270:372:1/18); 
\draw (320:1/16) node {$\theta_{1}^{1}$};

\draw (0.0834,0.02) arc (375:385:1/18); 
\draw (375:1/9) node {$\theta_{2}$};

\draw (0.315,0.65) arc (272:290:1/12); 
\draw (62:0.71) node {$\theta_{4}$};

\draw (0,0) node[left] {$B$};
\draw (0.003,0) node {$\bullet$};

\draw (0.316,0.68) node[right] {$E$};
\draw (0.316,0.68) node {$\bullet$};

\draw (0.011,1/6) node[left] {$C$};
\draw (0.011,1/6) node {$\bullet$};
\draw[every to/.style={append after command={[draw,dashed]}}] (0.73,1/6) to (0.011,1/6);

\draw (0.017,0.25) node[left] {$D$};
\draw (0.017,0.25) node {$\bullet$};
\draw[every to/.style={append after command={[draw,dashed]}}] (0.73,1/6) to (0.017,0.25);
\draw[every to/.style={append after command={[draw,dashed]}}] (0.017,0.25) to (0.15,0.25);
\draw (0.018,0.21) arc (270:338:1/18); 
\draw (73:1/5) node {$\theta_{1}^{2}$};
\draw[every to/.style={append after command={[draw,dashed]}}] (0.017,0.25) to (0.017,0.18);
\draw (0.1,0.239) arc (333:339:1/9); 
\draw (64:0.27) node {$\theta_{3}$};

\draw[every to/.style={append after command={[draw,dashed]}}] (0.73,1/6) to (0.009,0.1);
\draw (0.009,0.1) node[left] {$y=p(u)$};
\draw (0.009,0.1) node {$\bullet$};
\draw[every to/.style={append after command={[draw,dashed]}}] (0.009,0.04) to (0.009,0.1);
\draw (0.009,0.07) arc (270:337:1/18); 
\draw (55:0.09) node {$\theta_{1}(y)$};
\draw[every to/.style={append after command={[draw,dashed]}}] (0.1,0.15) to (0.009,0.1);
\draw (0.05,0.105) arc (337:356:1/18); 
\draw (52:0.15) node {$\theta_{2}(y)$};

\draw[every to/.style={append after command={[draw,dashed]}}] (0.73,1/6) to (0.05,0.4);
\draw (0.05,0.4) node[left] {$y$};
\draw (0.05,0.4) node {$\bullet$};
\draw[every to/.style={append after command={[draw,dashed]}}] (0.05,0.3) to (0.05,0.4);
\draw (0.05,0.35) arc (270:334:1/18); 
\draw (70:0.36) node {$\theta_{4}(y)$};
\draw[every to/.style={append after command={[draw,dashed]}}] (0.2,0.4) to (0.05,0.4);
\draw (0.15,0.365) arc (333:368:1/18); 
\draw (60:0.43) node {$\theta_{3}(y)$};

\draw (35:0.26) node {$B_{R_{1}}(B)$};
\draw (1/4,0) arc (0:90:1/4);

\end{tikzpicture}
\end{center}
\caption{The characteristic boundary $\widetilde{\partial O}$ in the $SIV$ model}
\label{FigSIV}
\end{figure}

 In order to obtain a parametrization of the characteristic boundary on the interval [0, 2], we make the change of variable $u=t/(1+t)$ and the $ODE$ \eqref{ODESIV} can be re-written
\begin{equation}\label{ODESIV1}
\begin{cases}
      \frac{dy_1}{du}(u)=\frac{1}{(1-u)^2}[(\beta-\mu-\gamma) y_1(u)-\beta(1-\chi) y_1(u) y_2(u)-\beta y_1^{2}(u)] \\
      \frac{dy_2}{dt}(u)=\frac{1}{(1-u)^2}[\eta-\eta y_1(u)-(\eta+\mu+\theta) y_2(u)-\chi\beta y_1(u) y_2(u)].
\end{cases}
\end{equation}

\begin{remarks}
\begin{itemize}
  \item The part of characteristic boundary $\widetilde{\partial O}_{1}$ which starts from the point $B$ (different to the origin) to the unstable endemic equilibrium $\tilde{z}$ of the dynamical system \eqref{ODESIV} has as parametrization the solution $p^{1}(u)=(p^{1}_{1}(u),p^{1}_{2}(u))_{0\leq u<1}$ of \eqref{ODESIV1}  with the initial condition $B$. 
  \item The part of characteristic boundary $\widetilde{\partial O}_{2}$ which starts from the extremity $E$ to the unstable endemic equilibrium $\tilde{z}$ of the dynamical system \eqref{ODESIV} (see Figure \ref{FigSIV}) has as parametrization the solution $p^{2}(u)=(p^{2}_{1}(u),p^{2}_{2}(u))_{0\leq u<1}$ of the  dynamical system defined by \eqref{ODESIV1} with as initial condition the point $E$.  \\
Thus the characteristic boundary admits as parametric curve $(p(u))_{0\leq u\leq 2}$ defined by
 \begin{equation*}
 p(u)=
\begin{cases}
      p^{1}(u) \quad \text{if}\quad 0\leq u<1\\
      \tilde{z}  \quad\text{if}\quad u=1  \\
      p^{2}(2-u) \quad\text{if} \quad 1< u\leq 2.
\end{cases}
\end{equation*} 
  \item As the tangent to the boundary $\widetilde{\partial O}$ at the origin $B$ is vertical, by the continuity there exist a ball $B_{R_{1}}(B)$ and a constant $\nu>0$ such that for all $(p_{1}(u),p_{2}(u))\in \widetilde{\partial O}\cap B_{R_{1}}(B)$, 
    \begin{equation*}
  \frac{\dot{p}_{2}(u)}{\dot{p}_{1}(u)}>\nu.
    \end{equation*}
\end{itemize}
In all what follows, $\mathcal{D}_{1.2}$ will be the line having the equation $z_{2}=\nu z_{1}$.
\end{remarks}
From these remarks we deduce that for all $t\in [0, 2]$, $\dot{p}_{1}(u)\geq0$ and $\dot{p}_{2}(u)\geq0$ since the first part of the characteristic boundary is below both $(\mathcal{D}_{1.1})$ and $\mathcal{H}_{1}$ and the second part $\widetilde{\partial O}_{2}$ is above both $(\mathcal{D}_{1.1})$ and $\mathcal{H}_{1}$. 

 Now we choose the point $z_{0}\in O$ such that the following conditions are satisfied
\begin{itemize}
  \item $z_{0}$ is above $(\mathcal{D}_{1.1})$ and $(\mathcal{H}_{1})$,
  \item $z_{0}$ is below  $(\mathcal{D}_{1.2})$,
  \item $z_{0}$ is at the right side of $(\mathcal{D}_{1.3})$,
  \item its second coordinate is smaller than that the point of $\widetilde{\partial O}$ at distance $R_1$ to $B$.
\end{itemize}
It clear that such a point exists.

 We now verify  the assumptions \ref{assumpinq} \ref{assump21} and \ref{assumpinq} \ref{assump22} through the proof of the following Lemma
\begin{lemma}
 The assumption \ref{assumpinq} \ref{assump21} is satisfied and there exists  $\theta\in]0,\pi[$ such that  for all $y\in\partial O$ and $a\in]0, 1[$,
\begin{equation}\label{assump222}
\dist(y^{a}, \partial O)\geq a\times\sin(\theta)\times \inf_{v\in \widetilde{\partial O}}|v-z_{0}|.
\end{equation} 
where $y^{a}=y+a(z_0-y)$ and $z_0$ is chosen above. 
\end{lemma}

\begin{proof}
 We first note that the only part of boundary $\partial O$  on which \eqref{assump222} could fail is the boundary $\widetilde{\partial O}$. So let's check it on this one. To this end we divide the boundary $\widetilde{\partial O}$ into three parts. The first one goes from $B$ to the intersection point (namely $C$) of $\widetilde{\partial O}$ and the horizontal line passing through the point $z_{0}$. The second part goes from $C$ to the intersection point (namely $D$) of $\widetilde{\partial O}$ and $B_{R_{1}}(B)$. And the last goes from $D$ to $E$.

  On the first part of $\widetilde{\partial O}$, we fix a current point $y$ and we denote by $\theta_{1}(y)$ the angle that the segment which joins the points $z_{0}$ and $y$ makes with the vertical line passing through $y$. We also denote $\theta_{2}(y)$ the angle that make the same segment  with the line with slope $\nu$ passing through $y$. It is not very difficult to see that there exist $\theta_{1}^{1}$ and $\theta_{2}$ such that $\pi/2\leq\theta_{1}(y)\leq\theta_{1}^{1}\leq\pi$, $0<\theta_{2}\leq\theta_{2}(y)\leq\pi/2$ and then $\sin(\theta_{1}(y))\geq\sin(\theta_{1}^{1})$ and $\sin(\theta_{2}(y))\geq\sin(\theta_{2})$(see figure  \ref{FigSIV}).
 
 On the part of the  boundary $\widetilde{\partial O}$ from $C$ to $D$, there exist $\theta_{1}^{2}$  such that $0<\theta_{1}^{2}\leq\theta_{1}(y)\leq\pi/2$ and $0<\theta_{2}\leq\theta_{2}(y)\leq\pi/2$ and then $\sin(\theta_{1}(y))\geq\sin(\theta_{1}^{2})$ and $\sin(\theta_{2}(y))\geq\sin(\theta_{2})$.
  
 For the part of $\widetilde{\partial O}$ from $D$ to $E$, the segment from $z_{0}$ to $y$ makes with the horizontal line passing through the point y an angle $\theta_{3}(y)$ and with the vertical line passing through the point $y$ an angle $\theta_{4}(y)$. Moreover it is not difficult to remark (see figure  \ref{FigSIV}) that  there exist $\theta_{3}$ and $\theta_{4}$ with $0\leq\theta_{3}\leq\theta_{3}(y)\leq\pi/2$, $0<\theta_{4}\leq\theta_{4}(y)\leq\pi/2$ such that $\sin(\theta_{3}(y))\geq\sin(\theta_{3})$ and $\sin(\theta_{4}(y))\geq\sin(\theta_{4})$.

 We deduce from the above that for all $y\in\widetilde{\partial O}$, 
\begin{align*}
\dist(y^{a}, \widetilde{\partial O})&\geq  \min_{i=1,2,3,4}\sin(\theta_{i}(y))\times |y^{a}-y| \\
&=a\times\min_{i=1,2,3,4}\sin(\theta_{i}(y))\times |y-z_{0}| \\
&=a\times\min_{i=1,2,3,4}\sin(\theta_{i})\times \inf_{v\in \widetilde{\partial O}}|v-z_{0}|.
\end{align*} 
\end{proof}

\subsection{A model with two levels of susceptibility $(S_0IS_1)$}
\subsubsection{Description of the model.}
\begin{figure} 

\begin{picture} (290,140)
\linethickness{0.4mm}

\put(330,90){\line(1,0){40}}
\put(330,90){\line(0,1){40}}
\put(370,90){\line(0,1){40}}
\put(330,130){\line(1,0){40}}

\put(60,90){\line(1,0){40}}
 \put(60,90){\line(0,1){40}}
\put(100,90){\line(0,1){40}}
\put(60,130){\line(1,0){40}}

\put(200,90){\line(1,0){40}}
\put(200,90){\line(0,1){40}}
\put(240,90){\line(0,1){40}}
\put(200,130){\line(1,0){40}}

\put(110,110){\vector(1,0){80}}
\put(250,125){\vector(1,0){65}}
\put(320,95){\vector(-1,0){65}}

\put(220,90){\vector(0,-1){40}}
\put(80,90){\vector(0,-1){40}}
\put(5,110){\vector(1,0){40}}
\put(350,90){\vector(0,-1){40}}

\put(70,100){\Huge{$S_0$}}
\put(215,100){\Huge{$I$}}
\put(340,100){\Huge{$S_1$}}

\put(120,115){\small{$\beta S^N_0 I^N/N$}}
\put(120,100){\small{(infection rate)}}

\put(245,115){\small{(recovered rate})}
\put(265,130){\small{$\alpha I^N$}}

\put(15,115){\small{$\mu N$}}
\put(5,100){\small{(birth rate)}}

\put(85,75){\small{$\mu S^N_0$}}
\put(85,65){\small{(death rate)}}

\put(230,70){\small{$\mu I^N$}}
\put(265,100){\small{$\tilde{\beta} S^N_1 I/N $}}

\put(240,85){\small{(re-infection rate)}}
\put(360,70){\small{$\mu S^N_1$}}

\end{picture}
\caption{Compartmental diagram of $S^N_0IS^N_1$ model.}
\label{CompS0IS1}
\end{figure}
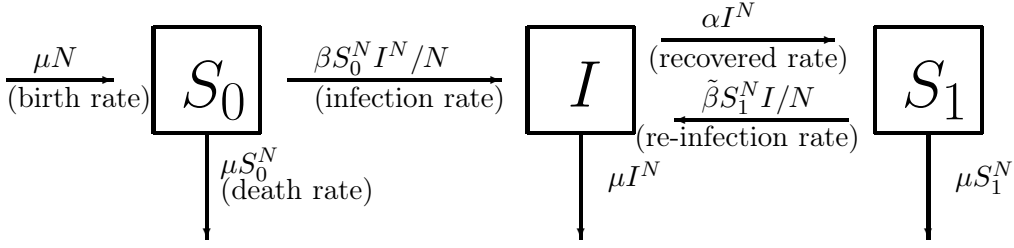

The $S^N_0I^NS^N_1$ epidemic model is a model which describes an endemic infection having two levels of susceptibility.
In this model, we suppose that a population with constant size $N$ is divided into three compartments namely: 
\begin{itemize}
  \item $S^{N}_0$-"the class of naive individuals" who are susceptibles without past infections. i.e individuals who have never been infected and may contract the infection.
  \item$ I^N$ is the class of infectious individuals
  \item $S^N_1$ is the class of susceptible individuals with at least one past infection. They are also called recovered.
\end{itemize}

 We assume that the probability of infection of an $S^{N}_0$ type (resp. $S^{N}_{1}$) individual upon contact with an infected individual is $r_0$ (resp. $r_1$). If $\kappa$ is the number of individuals which an infected individual meets per unit time, the rate of infection from $S_0$ (resp. $S_1$) compartment is $r_0\kappa S^{N}_0(t)I^{N}(t)/N$ (resp. $r_1\kappa S^{N}_0(t)I^{N}(t)/N$). Each infected individual recovers at rate $\alpha$. Each individual dies at rate $\mu$, at which he is replaced in the population by a susceptible individual. The total population size is constant equal to $N$.  The schematic representation of the disease transmission for this model is given by figure \ref{CompS0IS1}

\begin{figure}
\begin{center}
\begin{tikzpicture}[scale=15, line width=1pt]
\draw (0,0) -- (1,0);
\draw (0,0) -- (0,1);
\draw (1,0) -- (0,1);
\draw[draw=red, line width=1.2pt]  plot file {DataS0IS1.txt};
\draw (0,0) node[left] {$\bar{z}$};
\draw (0,0) node {$\bullet$};
\draw (0.0113606,0.68302727) node[right] {$\tilde{z}$};
\draw (0.0113606,0.68302727) node {$\bullet$};
\draw (0.14780605,0.81947272) node[above] {$z^{\ast}$};
\draw (0.14780605,0.81947272) node {$\bullet$};

\draw (0.5,-0.05) node{Proportion of infectious};
\draw (-0.05,0.5) node[rotate=90]{Proportion of susceptibles $S_{1}$};

\draw (71:0.9) node {$(\mathcal{H}_{2})$};
\draw [domain=0:0.1] plot(\x,{5*\x/(6*\x+0.015)});
\draw [domain=0.1:0.2] plot(\x,{5*\x/(6*\x+0.015)});
\draw [domain=0.2:0.5] plot(\x,{5*\x/(6*\x+0.015)});
\draw (76:1) node {$(\mathcal{D}_{2.1})$};
\draw [domain=0:0.4] plot(\x,{(3*\x+2.015)/3});

\draw[every to/.style={append after command={[draw,dashed]}}] (0.308,0) to (0.308,0.692);
\draw[every to/.style={append after command={[draw,dashed]}}] (0.308,0.692) to (0,0.692);
\draw (0.25,0.692) node[above] {$z_{0}$};
\draw (0.25,0.692) node {$\bullet$};

\draw (28:0.4) node {$B_{R_{3}}(B)$};
\draw (0.4746667,0) arc (0:180:1/6);
\draw (0.308,0) node[below] {$B$};
\draw (0.308,0) node {$\bullet$};

\draw (80:0.93) node {$B_{R_{2}}(F)$};
\draw (0,0.8) arc (270:315:0.2);
\draw (0,0.8) node[left] {$D$};
\draw (0,0.8) node {$\bullet$};

\draw (76:0.7) node {$(\mathcal{D}_{2.2})$};
\draw [domain=0:0.2] plot(\x,{(-2*\x+1});

\draw (48:0.35) node {$(\mathcal{D}_{2.3})$};
\draw [domain=0.2:0.308] plot(\x,{(-2.48139*\x+0.7642681});
\draw[every to/.style={append after command={[draw,dashed]}}] (0.308,0) to (0.25,0.692);

\draw[every to/.style={append after command={[draw,dashed]}}] (0.082,0.450) to (00.25,0.692);
\draw (0.082,0.450) node[left] {$y=q(u)$};
\draw (0.082,0.450) node {$\bullet$};
\draw[every to/.style={append after command={[draw,dashed]}}] (0.082,0.450) to (0.2,0.332);
\draw[every to/.style={append after command={[draw,dashed]}}] (0.082,0.450) to (0.082,0.6);
\draw (0.12,0.5) arc (50:90:1/18); 
\draw (77:0.55) node {$\theta_{1}(y)$};
\draw (0.11,0.42) arc (300:410:1/22); 
\draw (72:0.5) node {$\theta_{2}(y)$};

\draw[every to/.style={append after command={[draw,dashed]}}] (0.001,0.83) to (00.25,0.692);
\draw (0.001,0.83) node[left] {$y=q(u)$};
\draw (0.001,0.83) node {$\bullet$};
\draw[every to/.style={append after command={[draw,dashed]}}] (0.001,0.83) to (0.07,0.83);
\draw[every to/.style={append after command={[draw,dashed]}}] (0.001,0.83) to (0.02,0.786);

\draw[every to/.style={append after command={[draw,dashed]}}] (0.242,0.120) to (0.25,0.692);
\draw (0.242,0.120) node[left] {$y=q(u)$};
\draw (0.242,0.120) node {$\bullet$};
\draw[every to/.style={append after command={[draw,dashed]}}] (0.242,0.120) to (0.35,0.12);
\draw[every to/.style={append after command={[draw,dashed]}}] (0.242,0.120) to (0.19,0.2490323);

\draw[every to/.style={append after command={[draw,dashed]}}] (0,0.98) to (0.25,0.692);
\draw[every to/.style={append after command={[draw,dashed]}}] (0,0.98) to (0.1,0.88);
\draw[every to/.style={append after command={[draw,dashed]}}] (0,0.98) to (0.06,0.86);
\draw (0,0.98) node[left] {$y=q(u)$};
\draw (0,0.98) node {$\bullet$};
\draw[every to/.style={append after command={[draw,dashed]}}] (0,0.942) to (0.25,0.692);
\draw (0.001,0.942) node[left] {$E$};
\draw (0,0.942) node {$\bullet$};

\draw (0,1) node[left] {$F$};
\draw (0.001,0.998) node {$\bullet$};

\draw (0.225,0.145) node[left] {$C$};
\draw (0.225,0.145) node {$\bullet$};

\draw (0.308,0.692) node[right] {$G$};
\draw (0.308,0.692) node {$\bullet$};
\end{tikzpicture}
\end{center}
\caption{Characteristic boundary $\widetilde{\partial O}$ in the $S_{0}IS_{1}$ model}
\label{FigS0IS1}
\end{figure}

 The process described above is a continuous Markov Chain with state $(I^N(t), S^N_1(t))$. Let $P_{1},P_{2},P_{3},P_{4},P_{5}$ be the iid standard Poisson processes, we have
\begin{align}\label{modelS0IS1stoch}
     i^{N}(t) &=i^{N}(0)+\frac{1}{N}P_{1}\Big(N\int_{0}^{t}\beta i^{N}(u)(1-i^{N}(u)-s_1^{N}(u))du\Big)-\frac{1}{N}P_{2}\Big(N\int_{0}^{t}\alpha i^{N}(u)du\Big) \nonumber \\
     &-\frac{1}{N}P_{3}\Big(N\int_{0}^{t}\mu i^{N}(u)du\Big)+\frac{1}{N}P_{4}\Big(N\int_{0}^{t}r\beta i^{N}(u)s_1^{N}(u)du\Big) \\
      s_1^{N}(t)&=s_1^N(0)+ \frac{1}{N}P_{2}\Big(N\int_{0}^{t}\alpha i^{N}(u)du\Big)-\frac{1}{N}P_{4}\Big(N\int_{0}^{t}r\beta i^{N}(u)s_1^{N}(u)du\Big)-\frac{1}{N}P_{5}\Big(N\int_{0}^{t}\mu s_1^{N}(u)du\Big)  \nonumber,
\end{align}
where $i^N(t)=\frac{I^N(t)}{N}$ and $s_1(t)=\frac{S^N_1(t)}{N}$ represent respectively the proportion of infectious and the proportion of susceptibles which have already been sick.
Thus if we let  $Z^{N,z}(t)=(i^{N}(t),s_1^{N}(t))^{T}$ and
\begin{itemize}
  \item $\beta_{1}(x)=\beta x_{1}(1-x_{1}-x_{2})$, $h_{1}=(1,0)^{T}$
  \item $\beta_{2}(x)=\alpha x_{1}$, $h_{2}=(-1,1)^{T}$
  \item $\beta_{3}(x)=\mu x_{1}$, $h_{3}=(-1,0)^{T}$
  \item $\beta_{4}(x)=r\beta x_{1}x_{2}$, $h_{4}=(1,-1)^{T}$
  \item $\beta_{5}(x)=\mu x_{2}$, $h_{5}=(0,-1)^{T}$
\end{itemize}
The equation \eqref{modelS0IS1stoch} can be re-written as
\begin{equation}\label{PoissonS0IS1}
Z^{N,z_0}(t)=\frac{\left[N z_0\right]}{N}+\frac{1}{N}\sum_{j=1}^{5}h_{j}P_{j}\Big(N\int_{0}^{t}\beta_{j}(Z^{N,z_0}(s))ds\Big).
\end{equation}

 It is easy to show that for all $T>0$ the process $(Z^{N,z_0}(t))_{0\le t\le T}$ defined by \eqref{PoissonS0IS1} converges  almost surely and uniformly on $[0,T]$ as $N$ tends to $\infty$ to the solution $(i(t), s_1(t))$ of the $ODE$ 
\begin{equation}\label{ODES0IS1}
\begin{cases}
     \frac{dz_1}{dt}(t)=\beta (1-z_1(t)-z_2(t))z_1(t)-\mu z_1(t)-\alpha z_1(t)+r \beta z_1(t) z_2(t) \\
     \frac{dz_{2}}{dt}(t)=\alpha z_1(t)-\mu z_{2}(t)-r\beta z_1(t) z_{2}(t), 
\end{cases}
\end{equation}
with the initial condition $(i(0), s_1(0))=z_0$. Note that the deterministic model defined by \eqref{ODES0IS1} is the $S_0IS_1$ model studied by M.~Safan, H.~Heesterbeek, and K.~Dietz \cite{safan2006minimum}. The basic reproduction number is given by
\begin{equation}\label{nbaseS0IS1}
  R_{0}=\frac{\beta}{\alpha+\mu},
  \end{equation}
and for $r>1+\mu/\alpha$,  if we let
\begin{equation}\label{nbaseS0IS12}
  R^{\ast}_{0}=\frac{\beta^{\ast}}{\alpha+\mu}~\text{where}~\beta^{\ast}=\frac{(\sqrt{\mu(r-1)}+\sqrt{\alpha})^{2}}{r},
\end{equation}
then If the parameters values are chosen in such away that $R^{\ast}_{0}<R_{0}<1$ and $r>1+\mu/\alpha$, there exist two positive endemic equilibria ($EE$): the first one $z^{\ast}$ defined by \eqref{coordstable} which is locally asymptotically stable and the second one $\tilde{z}$ defined by \eqref{coordunstable} which is unstable. In addition the disease free equilibrium ($DFE$) $\bar{z}=(0,0)$ is locally asymptotically stable. 
\begin{equation}\label{coordstable}
 \begin{cases}
    z^{\ast}_{1}=\frac{1}{2}\Bigg(\bigg(1-\frac{1}{rR_{0}}-\frac{\mu}{(\alpha+\mu)R_{0}}\bigg)+\sqrt{\bigg(1-\frac{1}{rR_{0}}-\frac{\mu}{(\alpha+\mu)R_{0}}\bigg)^{2}+\frac{4\mu(1-\frac{1}{R_{0}})}{(\alpha+\mu)rR_{0}}}\Bigg)   \\
    z^{\ast}_{2}=\frac{\alpha z^{\ast}_{1}}{\mu+r\beta z^{\ast}_{1}}
    \end{cases}  
\end{equation}

\begin{equation}\label{coordunstable}
  \begin{cases}
    \tilde{z}_{1}&=\frac{1}{2}\Bigg(\bigg(1-\frac{1}{rR_{0}}-\frac{\mu}{(\alpha+\mu)R_{0}}\bigg)-\sqrt{\bigg(1-\frac{1}{rR_{0}}-\frac{\mu}{(\alpha+\mu)R_{0}}\bigg)^{2}+\frac{4\mu(1-\frac{1}{R_{0}})}{(\alpha+\mu)rR_{0}}}\Bigg)   \\
    \tilde{z}_{2}&=\frac{\alpha \tilde{z}_{1}}{\mu+r\beta \tilde{z}_{1}}  
   \end{cases}
\end{equation}

\subsubsection{The boundary $\widetilde{\partial O}$ in the $S_{0}IS_{1}$ model}
 In this section, we verify that the assumptions \ref{assumpinq} \ref{assump21} and \ref{assumpinq} \ref{assump22} are satisfied for the model $S^N_{0}I^NS_{1}^N$.
 On the figure \ref{FigS0IS1} we see the bassin of attraction $O$ of the equilibrium $z^{*}$ delimited by the boundary $\widetilde{\partial O}$ and containing the point $z^{*}$.
 
Let
\begin{itemize}
  \item $(\mathcal{D}_{2.1})$ the straight line having the equation $-(\alpha+\mu-\beta)+\beta(r-1) z_{2}-\beta z_{1}=0$ i.e
  \begin{equation*}
  z_{2}=\frac{1}{r-1}z_{1}+\frac{\alpha+\mu-\beta}{\beta(r-1)},
\end{equation*} 
\item $(\mathcal{H}_{2})$ the curve having the equation $\alpha z_{1}-\mu z_{2}-r\beta z_{1}z_{2}=0$ i.e
\begin{equation*}
 z_{2}=\frac{\alpha z_{1}}{r\beta z_{1}+\mu},
\end{equation*}
\end{itemize}

 In order to obtain a parametrization of the boundary $\widetilde{\partial O}$ on the interval [0, 2], we make the change of variable $u=t/(1+t)$ and the $ODE$ \eqref{ODES0IS1} can be re-written
\begin{equation}\label{ODES0IS12}
\begin{cases}
      \frac{dy_1}{du}(u)=\frac{1}{(1-u)^2}[(\beta-\mu) y_1(u)+\beta(r-1) y_1(u) y_2(u)-\beta y_1^{2}(u)] \\
      \frac{dy_2}{du}(u)=\frac{1}{(1-u)^2}[\alpha y_1(u)-\mu y_2(u)-r\beta y_1(u) y_2(u)].
\end{cases}
\end{equation}

\begin{remarks}
\begin{itemize}
  \item The part of boundary $\widetilde{\partial O}$ which starts from the point $B=(d,0)$ (intersection point between $\widetilde{\partial O}$ and the horizontal axis) to the unstable endemic equilibrium $\tilde{z}$ of the dynamical system \eqref{ODES0IS1} has as parametrization the solution $q^{1}(u)=(q^{1}_{1}(u),q^{1}_{2}(u))_{0\leq t<1}$ of \eqref{ODES0IS12}  with the initial condition $B=(d,0)$. 
  \item The part of boundary $\widetilde{\partial O}$  which starts from the unstable endemic equilibrium $\tilde{z}$ of dynamical system \eqref{ODES0IS1} to the point $F$ (different to (0,1)) is parametrized by the solution $q^{2}(u)=(q^{2}_{1}(u),q^{2}_{2}(u))_{0\leq u<1}$ of the $OD$E defined by \eqref{ODES0IS12}
 with initial condition the point $F$ (see Figure \ref{FigS0IS1}). \\
Thus the characteristic boundary admits as parametric curve $(q(u))_{0\leq u\leq 2}$ defined by
 \begin{equation*}
 q(u)=
\begin{cases}
      q^{1}(u) \quad \text{if}\quad 0\leq u<1\\
      \tilde{z} \quad \text{if}\quad u=1 \\
      q^{2}(2-u) \quad\text{if} \quad 1<u\leq 2.
\end{cases}
\end{equation*} 

  \item As the tangent to the boundary $\widetilde{\partial O}$ at the point $F$ is almost vertical, there exist a ball $B_{R_{2}}(F)$ and a constant $\omega>1$ such that for all $(q_{1}(u),q_{2}(u))\in \widetilde{\partial O}\cap B_{R_{2}}(F)$, 
    \begin{equation*}
  \frac{\dot{q}_{2}(u)}{\dot{q}_{1}(u)}<-\omega.
    \end{equation*}
  We then defined by $(\mathcal{D}_{2.2})$  the line having for equation $z_{2}=-\omega z_{1}+1$.
    
  \item The tangent to the boundary $\widetilde{\partial O}$ at the point $B=(d,0)$ is a line whose slope is bounded as follows
\begin{equation*}
-\frac{\alpha}{\alpha+\mu-\beta}<\frac{\dot{q}_{2}}{\dot{q}_{1}}.
\end{equation*} 
As $\widetilde{\partial O}$ is a continuous curve, there exists a ball $B_{R_{3}}(B)$ such that for all $(q_{1}(u),q_{2}(u))\in \widetilde{\partial O}\cap B_{R_{3}}(B)$, 
\begin{equation*}
-\frac{\alpha}{\alpha+\mu-\beta}<\frac{\dot{q}_{2}}{\dot{q}_{1}}.
\end{equation*} 
We also defined by $(\mathcal{D}_{2.3})$  the line having for equation $z_{2}=-\frac{\alpha}{\alpha+\mu-\beta } z_{1}+\frac{\alpha d}{\alpha+\mu-\beta }$.
\end{itemize}
\end{remarks}

 From these remarks we deduce that for all $(q_{1}(u),q_{2}(u))\in\widetilde{\partial O}$, $\dot{q}_{1}(u)\leq0$ and $\dot{q}_{2}(u)\geq0$ since the first part of the  boundary $\widetilde{\partial O}$ is below of $(\mathcal{D}_{2.1})$ and $\mathcal{H}_{2}$ and the second part is above of $(\mathcal{D}_{2.1})$ and $\mathcal{H}_{2}$. 

 Now we choose the point $z_{0}\in \mathring{O}$ such that the following conditions are satisfied
\begin{itemize}
  \item $z_{0}$ is the below of $(\mathcal{D}_{2.1})$ and $(\mathcal{H}_{2})$,
  \item $z_{0}$ is on right of $(\mathcal{D}_{2.2})$ and $(\mathcal{D}_{2.3})$,
  \item $z_{0}$ is a point of the horizontal line passing through $G$ (intersection point between the vertical line passing through $B$ and the line whose equation is $z_1+z_2=1$), 
  \item the orthogonal projection of $z_{0}$ on the horizontal axis is at a distance smaller than $R_{3}$ from point $B$,
  \item the projection  of $z_{0}$ on the vertical axis and parallel to the line whose equation is $z_1+z_2=1$ is at a distance smaller than $R_{2}$ from point $F$.
\end{itemize}

 We rewrite the ODE \eqref{ODES0IS1} as $\dot{z}(t)=g(z(t))$. We want to verify that for any $z\in\widetilde{\partial O}$ with $z_2<\tilde{z}_2$,
the vector $g(z)$ points in the sector $(3\pi/4,\pi)$. The reader can then easily verify that the same is true for $-g(z)$ if
$z\in\widetilde{\partial O}$ while $z_2>\tilde{z}_2$.

 The fact that $g(z)$ points in the north--west direction follows from $g_1(z)<0$, $g_2(z)>0$. It thus remains to prove the
\begin{lemma}
For any $z\in\widetilde{\partial O}$ with $z_2<\tilde{z}_2$, $g_1(z)+g_2(z)>0$.
\end{lemma}
\noindent{\sc Proof} 
$\{z(t),\ t\ge0\}$ denoting any trajectory of the ODE  \eqref{ODES0IS1}, we define $\xi(t)=z_1(t)+z_2(t)$. It is easy to verify that
\[\frac{d}{dt}\dot{\xi}(t)=-(\mu+\beta z_1(t))\dot{\xi}(t)+\beta(1-\xi(t))\dot{z}_1(t),\]
hence for any $0\le t_0<t$,
\[\dot{\xi}(t)=\dot{\xi}(t_0)\exp\left(-\int_0^t(\mu+\beta z_1(s))ds\right)+\beta\int_0^t\exp\left(-\int_s^t(\mu+\beta z_1(r))dr\right)(1-\xi(s))\dot{z}_1(s)ds,\]
and since $g_1(z)<0$ if $z\in\widetilde{\partial O}$ with $z_2<\tilde{z}_2$, we have that along the trajectory of  \eqref{ODES0IS1} from $B$ to $\tilde{z}$, 
if $\dot{\xi}(t_0)\le0$, then $\dot{\xi}(t)<0$ for all $t>t_0$.

 Now, at a point $z=(\tilde{z}_1+a,\tilde{z}_2-a)$ for $a\in\mathbb{R}$, we have that
\[\dot{\xi}(t)=(\alpha+\mu-r\beta \tilde{z}_2)a.\]
At the point $(0,\tilde{z}_1+\tilde{z}_2)$ (that is with $a=-\tilde{z}_1$), we note that $\dot{z}_1(t)=0$ and $\dot{z}_2(t)=-\mu z_2<0$, 
which implies that $(\alpha+\mu-r\beta \tilde{z}_2)>0$. Consequently at any point $z$ on the half line $\{(\tilde{z}_1+a,\tilde{z}_2-a),\ a>0\}$,
$\dot{\xi}(t)>0$. 

 We first conclude from the above two facts that the trajectory of  \eqref{ODES0IS1} from $B$ to $\tilde{z}$ lies entirely below the half line
$\{(\tilde{z}_1+a,\tilde{z}_2-a),\ a>0\}$. Indeed, if that would not be the case, there would be points above that half--line which would be 
on the left of $\widetilde{\partial O}$, hence a trajectory of  \eqref{ODES0IS1} starting from such a point would eventually converge to $(0,0)$, hence cross downward the half line $\{(\tilde{z}_1+a,\tilde{z}_2-a),\ a>0\}$, which is impossible.

 Since the trajectory of  \eqref{ODES0IS1} from $B$ to $\tilde{z}$ lies entirely below the half line
$\{(\tilde{z}_1+a,\tilde{z}_2-a),\ a>0\}$, necessarily for any $t_0>0$ there exists $t>t_0$ such that $\dot{\xi}(t)>0$, which, as a consequence of the first statement in the present proof, implies that $\dot{\xi}(t)>0$ for any $t>0$, hence the result.

 We now deduce the assumption \ref{assumpinq} \ref{assump21} and  \ref{assumpinq} \ref{assump22} through the proof of the following Lemma. 

\begin{lemma}
The assumption \ref{assumpinq} \ref{assump21} is satisfied and there exist $\theta\in]0,\pi[$ such that for all $y\in\partial O$
\begin{equation*}
\dist(y^{a}, \partial O)\geq a\times\sin(\theta)\inf_{v\in \partial O} |v-z_{0}|,
\end{equation*}
where $y^{a}=y+a(z_0-y)$ and $z_0$ is chosen above.
\end{lemma}

\begin{proof}
We note here that it is enough to show this last inequality over $\widetilde{\partial O}$.
Now for $y\in\widetilde{\partial O}$, if $y\in\widetilde{\partial O}\cap B_{R_{3}}(B)$,  we define $\theta_{2}(y)$ as the angle made by the horizontal line passing through $y$ and the vector from $y$ to $z_{0}$ and $\theta_{1}(y)$ as the angle made by the vector from $y$ to $z_{0}$ and the parallel line to $(\mathcal{D}_{2.3})$  passing through $y$. For the part of $\widetilde{\partial O}$ from $C$ to $D$, $\theta_{2}(y)$ is the angle made by the parallel line to the line whose equation is $z_1+z_2=1$ passing through $y$ and the vector from $y$ to $z$. The angle $\theta_{1}(y)$ is made by the vector from $y$ to $z$ and the vertical line passing through $y$. For the part of $\widetilde{\partial O}$ from $D$ to $E$, $\theta_{1}(y)$  is made by the vector from $y$ to $z_{0}$ and the horizontal line passing through $y$ and $\theta_{2}(y)$ is the angle between the parallel line to $(\mathcal{D}_{2.2})$ and the vector from $y$ to $z_{0}$. In the last part of $\widetilde{\partial O}$ i.e from $E$ to $F$, $\theta_{2}(y)$ is made as in the part from $D$ to $E$ and $\theta_{1}(y)$ is the angle between the vector from $y$ to $z_{0}$ and the parallel line to the second bisector passing through $y$. It is easy to see in each part of $\widetilde{\partial O}$ that there exists $\theta\in]0,\theta[$ such that $\sin(\theta_{1}(y)), \sin(\theta_{2}(y))\geq\sin(\theta)$. And then  there exists $\theta\in]0,\pi[$  such that 
\begin{align*}
\dist(y^{a}, \widetilde{\partial O})&\geq \min_{i=1,2}\sin(\theta_{i}(y))\times |y^{a}-y| \\
&=a\times \min_{i=1,2}\sin(\theta_{i}(y))\times |y-z_{0}|\\
&\geq a\times\sin(\theta)\inf_{v\in \widetilde{\partial O}} |v-z_{0}|.
\end{align*} 
\end{proof}

\frenchspacing
\bibliographystyle{plain}

\end{document}